\documentclass{amsart}

\usepackage{geometry,graphicx,amssymb,amsmath,amsbsy,eucal,amsfonts,mathrsfs,amscd,bm,tcolorbox, xcolor,caption,float,tikz,tikz-3dplot,cancel}
\usepackage[font=small]{subcaption}
\usetikzlibrary{patterns}
\usepackage[textsize=tiny,color=blue!50!white]{todonotes}
\usepackage[normalem]{ulem} 

\usepackage[breaklinks,bookmarks=false]{hyperref}

\hypersetup{colorlinks, linkcolor=blue, citecolor=blue,
urlcolor=blue, plainpages=false, pdfwindowui=false,
pdfstartview={FitH}, pdftitle={},}

\numberwithin{equation}{section}

\allowdisplaybreaks[3]

\newtheorem{theorem}{Theorem}[section]
\newtheorem{lemma}[theorem]{Lemma}
\newtheorem{assumption}[theorem]{Assumption}
\newtheorem{corollary}[theorem]{Corollary}
\newtheorem{proposition}[theorem]{Proposition}
\theoremstyle{definition}

\newtheorem{remark}[theorem]{Remark}

\usepackage{mydef}

\begin{document}

\title[Discrete-trace preserving commuting projections]{Bounded, Commuting, Discrete-trace Preserving Projections}

\def\AEr{Alexandre Ern}
\def\AEa{CERMICS, CNRS, ENPC, Institut Polytechnique de Paris, 77455 Marne-la-Vall\'ee, France \& Inria Paris, 48 rue Barrault, 75647 Paris, France}

\def\JG{Johnny Guzm\'an}
\def\JGa{Division of Applied Mathematics,
Brown University,
Box F,
182 George Street,
Providence, RI 02912, USA}

\def\PP{Pratyush Potu}
\def\PPa{Division of Applied Mathematics,
Brown University,
Box F,
182 George Street,
Providence, RI 02912, USA}

\def\MV{Martin Vohral\'ik}
\def\MVa{Inria Paris, 48 rue Barrault, 75647 Paris, France \& CERMICS, ENPC, Institut Polytechnique de Paris and CNRS, 77455 Marne-la-Vall\'ee, France}

\author{\AEr}
\address{\AEa}
\email{\href{mailto:alexandre.ern@enpc.fr}{alexandre.ern@enpc.fr}}

\author{\JG}
\address{\JGa}
\email{\href{mailto:johnny_guzman@brown.edu}{johnny\_guzman@brown.edu}}

\author{\PP}
\address{\PPa}
\email{\href{mailto:pratyush_potu@brown.edu}{pratyush\_potu@brown.edu}}


\thanks{The authors are thankful to Martin Vohral{\'i}k (INRIA Paris) for stimulating discussions on the applications of the discrete-trace preserving projections discussed in Section~\ref{sec:projections}.}

\begin{abstract} 
We construct bounded, commuting projections for the three-dimensional de Rham complex with the additional property that the projections preserve the trace of functions/fields if the latter is a piecewise polynomial in the appropriate trace space. The projections are locally defined and stable in the graph norm. More precisely, the part of the graph norm involving the exterior derivative only involves the oscillation of this derivative in a narrow strip of elements touching the boundary and weighted by the local mesh size. Moreover, the projections are $L^2$-stable locally when acting on functions/fields whose exterior derivative is a piecewise polynomial in the appropriate space. We present two salient applications of the present bounded, commuting, discrete-trace preserving projections: the construction of stable liftings of piecewise polynomial data and an optimality result on the discrete versus continuous extension of piecewise polynomial data.
\end{abstract}


\maketitle
    


\section{Introduction}

In \cite{EGPV_HO:24}, extending previous ideas from~\cite{Arn_Guz_loc_stab_L2_com_proj_21}, locally defined, $L^2$-stable, commuting projections are constructed for the finite element complex of Whitney forms in three dimensions both without boundary conditions and with homogeneous boundary conditions. In the present work, we develop locally defined, commuting projections which satisfy the additional property that a function/field with piecewise polynomial trace will be a function/field in the appropriate discrete space with the same piecewise polynomial trace. Because of this additional property, we call our new operators discrete-trace preserving projections. Such projections cannot be $L^2$-stable in general, but the projections we construct are stable in the graph norm. Moreover, they are $L^2$-stable when restricted to functions/fields whose exterior derivative is a piecewise polynomial. In the general case, the part of the graph norm involving the exterior derivative only involves the oscillation of this derivative (measured by its departure from the corresponding piecewise polynomial space) and is weighted by the local mesh size.

The property of preservation of piecewise polynomial boundary data is also observed in the Scott-Zhang interpolant \cite{Scott_Zhang_90}. As noted in \cite[Section 5]{Scott_Zhang_90}, an application of this property is the construction of a discrete lifting of the piecewise polynomial boundary data. Essentially, by composing a continuous level extension (e.g. the harmonic extension) with a discrete-trace preserving interpolant, one obtains a bounded discrete lifting. This lifting then delivers error estimates for nonhomogeneous problems in terms of the best-approximation error from the discrete space and the approximation of the boundary data separately, see \cite[Section 1]{Ains_Guzm_Saya_16}. One salient outcome of the present construction is therefore to generalize the above results to any discrete space from the finite element complex of Whitney forms in a systematic way. 
Another salient outcome of our construction is to establish an optimality result on the discrete versus continuous extension of piecewise polynomial data.
Previous bounded commuting projections have been developed in \cite{Arn_Guz_loc_stab_L2_com_proj_21, Chaum_Voh_H_curl_proj_24, Ern_Gud_Sme_Voh_loc_glob_div_22, ErnGu:16_molli, EGPV_HO:24, Falk_Winth_loc_coch_14, Gawl_Hols_Lich_21, Hipt_Pech_disc_1fo_19}, extending previous work in \cite{Christ_Wint_sm_proj_08, schoberlmultilevel}. Preservation of \textit{homogeneous} boundary conditions is explicitly discussed in \cite{Chaum_Voh_H_curl_proj_24, Christ_Wint_sm_proj_08, Ern_Gud_Sme_Voh_loc_glob_div_22, ErnGu:16_molli, EGPV_HO:24, Gawl_Hols_Lich_21, Hipt_Pech_disc_1fo_19}. The present work extends this previous work by additionally preserving discrete boundary data.

The present construction hinges on three main ideas. For the sake of clarity, the presentation focuses on the three-dimensional setting, i.e., the ambient space dimension is $d=3$. The first key idea is to combine the commuting operators from \cite{EGPV_HO:24} preserving homogeneous boundary conditions with some novel operators that commute under the trace operator and that are locally defined near the boundary. We call these novel operators $\{\Pib^l\}_{l\in\{0{:}3\}}$. The main result concerning these operators is Theorem~\ref{thm:Pi}. The second key idea, shared with~\cite{EGPV_HO:24}, is to build the operators $\{\Pib^l\}_{l\in\{0{:}2\}}$ using a two-scale decomposition into lowest-order and higher-order operators, which we call $\{P_0^l\}_{l\in\{0{:}2\}}$ and $\{P_+^l\}_{l\in\{0{:}2\}}$. (The construction of $\Pib^3$ does not require this decomposition.) The main result on the two-scale decomposition is Theorem~\ref{thm:Satisfaction_of_Pib}. The third key idea is to define the operators $\{P_0^l\}_{l\in\{0{:}2\}}$ and $\{P_+^l\}_{l\in\{0{:}2\}}$ by means of boundary weights satisfying suitable properties. More precisely, these boundary weights are first extended into the bulk, locally near the boundary, and then the above operators are defined by suitable linear forms based on integration by parts formulae inside the bulk, see~\eqref{eq:alpha_def} and~\eqref{eq:def_P0+}. The main results on the boundary weights are Lemma~\ref{lem:cond_zeta0} and Lemma~\ref{lem:cond_zeta+}. The construction of the lowest-order boundary weights follows the general idea from \cite{EGPV_HO:24} in the bulk, but requires nontrivial modifications. Indeed, the lowest-order boundary weights are still constructed using solutions of local problems. However, in the present work, these local problems are solved over the boundary and involve surface differential operators. We remark that the idea of extending weight functions defined on the boundary to the interior was also used in \cite{Gawl_Hols_Lich_21}.

The paper is organized as follows. In Sections \ref{sec:cont} and \ref{sec:discrete}, we introduce the continuous and discrete settings, respectively. Then, in Section \ref{sec:projections}, we give the main idea of the construction of the discrete-trace preserving commuting projections. Also in Section \ref{sec:projections}, we showcase several important applications of these projections. The rest of the paper is devoted to their construction.  In particular, in Section \ref{sec:construction_Pib}, we introduce the two main ideas of the construction, namely a two-level decomposition involving lowest-order and higher-order components on the one hand and the use of boundary weights together with integration by parts formulae inside the domain on the other hand. The lowest-order part of the construction is detailed in Section \ref{sec:lowest_order} and the higher-order part in Section \ref{sec:high_order}.


\section{Continuous setting}\label{sec:cont}

In this section, we present the continuous setting.

\subsection{Bulk setting}
We use much of the notation of \cite{EGPV_HO:24}. In particular, boldface font is used for vectors, vector fields, and functional spaces composed of such fields.
Let $\Omega$ be a Lipschitz, polyhedral, open, bounded, connected set in $\RRR^3$. 
We consider the de Rham complex
\begin{alignat}{4}\label{complex}
&\mathbb{R}
\stackrel{\subset}{\xrightarrow{\hspace*{0.5cm}}}\
 V^0
&&\stackrel{d^0}{\xrightarrow{\hspace*{0.5cm}}}\
 \bV^1
&&\stackrel{d^1}{\xrightarrow{\hspace*{0.5cm}}}\
 \bV^2
&&\stackrel{d^2}{\xrightarrow{\hspace*{0.5cm}}}\
V^3
\stackrel{d^3}{\xrightarrow{\hspace*{0.5cm}}}
0,
\end{alignat}
where 
\begin{equation}\label{eq:differential_operators}
    d^0 := \grad,\quad d^1 := \curl,\quad d^2 := \div,\quad d^3 := 0,
\end{equation}
and the relevant graph spaces are
\begin{equation} \label{eq:graph_spaces}
V^0:= H(\grad,\Omega) =H^1(\Omega), \quad \bV^1:= \bH(\curl, \Omega), \quad \bV^2:= \bH(\dive, \Omega), \quad V^3:= L^2(\Omega). 
\end{equation}
These spaces are equipped with the canonical graph norm which we denote $\|\SCAL\|_{V^l}$. We generically denote the $L^2$- or $\bL^2$-inner product over $\Omega$ by $(\cdot, \cdot)_{\Omega}$.

\subsection{Subspaces with zero trace}\label{subsec:homogeneous_spaces}

We denote by $\bn$ the unit outward normal to $\Omega$ and $\Gamma := \partial\Omega$. We now define the following trace operators for the graph spaces:
\begin{equation}\label{eq:def_trace}
    \tr^0(u) := u|_{\Gamma}, \quad \tr^1(\bu) := (\bn \times \bu|_{\Gamma}) \times \bn, \quad \tr^2(\bu) := \bu|_{\Gamma} \cdot \bn.
\end{equation}
We will also need the \emph{twisted tangential trace}, denoted by $\tr_{\perp}^1$, which is defined as
\begin{equation}
    \tr_{\perp}^1(\bu) := \bn \times \tr^1(\bu) = \bn \times \bu|_{\Gamma}.
\end{equation}
While not exactly a trace operator, we set
\begin{equation}
    \tr^3(u) := \int_\Omega u.
\end{equation}
This choice of notation for $\tr^3$ is not standard, but allows for a unified presentation of our results.
The graph spaces with zero boundary conditions are, for all $l\in\{0{:}3\}$,
\begin{equation} 
\mV^l:= \{u\in V^l: \tr^l(u)=0\}.
\end{equation}%
These subspaces satisfy the zero-trace counterpart of~\eqref{complex}, up to the change $d^3:=\tr^3$.

\subsection{Surface differential operators}
We introduce continuous surface spaces and surface differential operators. We take the point of view that the surface vector fields are $\R^3$-valued. For instance, recalling that $\Gamma = \partial\Omega$ and $\bn$ is the unit outward normal to $\Omega$, the space of $\bL^2$-integrable tangential vector fields over $\Gamma$ is
\begin{equation}
    \bL^2_\textsc{t}(\Gamma) := \{\bu \in \bL^2(\Gamma) : \bu \cdot \bn = 0 \}.
\end{equation}

We denote by $\sgrad$ the surface gradient on $\Gamma$. For a scalar function $u$, defined on $\Gamma$, $\sgrad(u)$ can be explicitly defined in terms of local coordinates or as local projections of $\grad(\tilde{u})$ onto $\Gamma$, where $\tilde{u}$ is an extension of $u$ to a neighborhood of $\Gamma$ (see, e.g., \cite[Definition 3.1]{Buffa02}). We let 
\begin{equation}
    H^1(\Gamma) := \{u \in L^2(\Gamma) : \sgrad(u) \in\bL_{\textsc{t}}^2(\Gamma)\}.
\end{equation}
We then define the surface vector curl (rotated gradient) operator $\srot: H^1(\Gamma) \to \bL^2_{\textsc{t}}(\Gamma)$ as 
\begin{equation}
    \srot(u) := \bn \times \sgrad(u).
\end{equation}
We denote
\begin{equation}
    H(\sgrad,\Gamma) = H(\srot,\Gamma) := H^1(\Gamma).
\end{equation}


We generically denote the $L^2$- or $\bL^2$-inner product over $\Gamma$ by $(\cdot, \cdot)_{\Gamma}$. The surface scalar curl and surface divergence are denoted  $\scurl$ and $\sdive$, respectively, and the domains of these operators are subspaces of $\bL^2_{\textsc{t}}(\Gamma)$. Explicit definitions can be found in \cite[Chapter 3.4]{Monk_03} or \cite{Buffa02} for instance. Using these operators, we define the spaces
\begin{subequations}
\begin{alignat}{2}
    \bH(\scurl, \Gamma) &:= \{\bu \in \bL^2_{\textsc{t}}(\Gamma): \scurl (\bu) \in L^2(\Gamma)\}, \\
    \bH(\sdive, \Gamma) &:= \{\bu \in \bL^2_{\textsc{t}}(\Gamma): \sdive (\bu) \in L^2(\Gamma)\}.
\end{alignat}
\end{subequations}
These spaces form the complexes:
\begin{subequations}
\begin{alignat}{3}
&\mathbb{R}
\stackrel{\subset}{\xrightarrow{\hspace*{0.5cm}}}\
 H(\sgrad,\Gamma)
&&\stackrel{\sgrad}{\xrightarrow{\hspace*{0.5cm}}}\
 \bH(\scurl, \Gamma)
&&\stackrel{\scurl}{\xrightarrow{\hspace*{0.5cm}}}\
 L^2(\Gamma)
\stackrel{\int_\Gamma}{\xrightarrow{\hspace*{0.5cm}}}
0,\label{eq:surface_complex_curl}\\
&\mathbb{R}
\stackrel{\subset}{\xrightarrow{\hspace*{0.5cm}}}\
 H(\srot,\Gamma)
&&\stackrel{\srot}{\xrightarrow{\hspace*{0.5cm}}}\
 \bH(\sdive, \Gamma)
&&\stackrel{\sdive}{\xrightarrow{\hspace*{0.5cm}}}\
 L^2(\Gamma)
\stackrel{\int_\Gamma}{\xrightarrow{\hspace*{0.5cm}}}
0.\label{eq:surface_complex_div}
\end{alignat}
\end{subequations}
Moreover, the following identities hold true:
\begin{subequations}
    \begin{alignat}{3}
        (\scurl(\bv), w )_\Gamma &= (\bv, \srot(w))_\Gamma, \quad &&\forall \bv \in \bH(\scurl, \Gamma), \;&&\forall w \in H^1(\Gamma),\label{eq:scurl_adjoint}\\
        (\sdive(\bv), w )_\Gamma &= - (\bv, \sgrad(w))_\Gamma, \quad &&\forall \bv \in \bH(\sdive, \Gamma),\;&&\forall w \in H^1(\Gamma).\label{eq:sdive_adjoint}
    \end{alignat}
\end{subequations}
We note the following commuting relationships between the trace operators defined in \eqref{eq:def_trace} and the surface differential operators:
\begin{subequations}\label{eq:trace_d_commuting} \begin{alignat}{2}
\sgrad (\tr^0(u)) &= \tr^1 (\grad(u)), &\quad \scurl (\tr^1(\bv)) &= \tr^2 (\curl( \bv)),\\
\srot (\tr^0(u)) &= \tr^1_\perp (\grad(u)), &\quad \sdive (\tr^1_\perp(\bv)) &= -\tr^2 (\curl( \bv)),
\end{alignat} \end{subequations}
which hold almost everywhere for smooth enough $u$ and $\bv$.

\section{Discrete setting}
\label{sec:discrete}

In this section, we present the discrete setting. 

\subsection{Simplicial mesh and geometric objects}
Let $\Th$ be a simplicial triangulation of $\Omega$.
The shape-regularity parameter of the mesh $\Th$ is defined as 
\begin{equation} \label{eq_rho}
\rho_{\Th} \eq \max_{\tau \in \Th} h_\tau / \theta_\tau,
\end{equation}
where $h_\tau$ is the diameter of $\tau$ and $\theta_\tau$ the diameter of the largest ball inscribed in $\tau$.

For all $l\in\{0{:}3\}$, the $l$-simplices in $\Th$ are the mesh vertices for $l=0$,
the mesh edges for $l=1$, the mesh faces for $l=2$, and the mesh tetrahedra for $l=3$.
Notice that $l$-simplices are, by definition, closed sets.
We enumerate the vertices of $\Th$ and denote the set of vertices by $\Vh:=\{ x_0, \ldots, x_{N}\}$.
All the $l$-simplices are oriented by taking their vertices in
increasing enumeration order \cite[Sec.~10.3]{ErnGuermondbook}.
We denote the collection of (oriented) $l$-simplices as
\begin{equation} \label{eq:def_sigma}
\Delta_h^l := \{ \sigma = [x_{i_0},\ldots,x_{i_l}] \,:\, 0\le i_0<\ldots<i_l\le N\},
\end{equation}
where the brackets denote the convex hull of a set of points.
A more explicit notation is
\begin{equation}
\Vh := \Delta_h^0, \quad \Eh := \Delta_h^1, \quad \Fh := \Delta_h^2, \quad \Th := \Delta_h^3,
\end{equation}
where $\Eh$ is the set of (oriented) mesh edges, $\Fh$ is the set of (oriented) mesh faces, and $\Th$ the set of (oriented) mesh cells
(tetrahedra). 
The set $\Delta_h := \bigcup_{l\in\{0{:}3\}} \Delta_h^l$ is the collection of all the (oriented) geometric objects in the mesh. 
For every edge $e:=[x_{i_0}, x_{i_1}] \in \Eh$, we let $\bt_e$ be the unit tangent vector to $e$ pointing from $x_{i_0} $ to $x_{i_1}$. 
For every face $f:= [x_{i_0}, x_{i_1}, x_{i_2}] \in \Fh$, we let $\bn_f$ be the unit normal vector to $f$ such that $\bn_f:=\bt_{e_1}{\times}\bt_{e_2}$
with $e_1:=[x_{i_0}, x_{i_1}]$ and $e_2:=[x_{i_0}, x_{i_2}]$.

We will need to consider certain patches of tetrahedra. For all $l\in\{0{:}3\}$ and all $\sigma \in \Delta_h^l$, we define the \emph{star} of $\sigma$ as
\begin{equation}
    \st(\sigma) := \inte \bigg( \bigcup_{\substack{\tau \in  \Th \\\sigma \subset \tau}} \tau \bigg),
\end{equation}
where $\inte(\cdot)$ denote the interior of a subset set in $\R^3$.
For example, if $\sigma$ is a tetrahedron, $\st(\sigma)=\inte(\sigma)$.
Next, we define the $k$-th order extended star recursively as
\begin{equation}
\es^1(\sig) := \inte \bigg(
\bigcup_{\substack{\tau\in\Th \\ \sigma \cap \tau \neq \emptyset}}\tau\bigg), \qquad 
\es^k(\sigma) := \inte \bigg( \bigcup_{\substack{\tau \in \Th \\ \tau \cap \clos(\es^{k-1}(\sigma))\neq \emptyset}} \tau\bigg),\:\: \forall k\geq 2,
\end{equation}
where $\clos(\cdot)$ denotes the closure of a subset of $\R^3$. 
An equivalent definition for $\es^k(\sigma)$ is the union of $\st(v)$ for all the vertices $v\in \clos(\es^{k-1}(\sigma))$ for any $k\geq 1$ with the additional convention $\es^0(\sigma) = \sigma$. We simply write $\es(\sigma)$ for $\es^1(\sigma)$.  Also, note that if $\sigma$ is a vertex, then $\es(\sigma)=\st(\sigma)$. 
We define $h_{\sigma}:= \diam(\sigma)$ if $l \ge 1$ and $h_{\sigma}:=\diam(\st(\sigma))$ if $l=0$.

The sets of geometric objects lying on the boundary and in the interior of $\Omega$ are 
\begin{equation}
    \Delta_h^{l,\partial} := \{\sigma \in\Delta_h^l: \sigma\subset\Gamma\}, \qquad 
    \mathring{\Delta}_h^l := \Delta_h^l\setminus \Delta_h^{l,\partial}, \quad 
    \forall l\in\{0{:}3\}.
\end{equation}
Notice that $\Delta_h^{3,\partial}=\emptyset$ and $\mathring{\Delta}_h^3 = \Delta_h^3$. We set $\Delta_h^\partial = \bigcup_{l\in\{0{:}2\}} \Delta_h^{l, \partial}$ and $\mathring{\Delta}_h = \Delta_h \setminus \Delta_h^\partial$. We also use the more explicit notation
\begin{equation}
\Vh^\partial:=\Delta_h^{0,\partial}, \quad \Eh^\partial:=\Delta_h^{1,\partial}, \quad 
\Fh^\partial:=\Delta_h^{2,\partial}, \quad \mcalVh := \mathring{\Delta}_h^0, \quad 
\mEh := \mathring{\Delta}_h^1, \quad \mFh := \mathring{\Delta}_h^2.
\end{equation}

For all $l\in\{0{:}2\}$ and all $\sigma\in\Delta_h^{l,\partial}$, we define the boundary star and the extended boundary star of $\sigma$ as
\begin{equation} \label{eq:def_esb}
    \stb(\sigma) := \relinte \bigg( \bigcup_{\substack{f \in  \Fh^\partial \\ \sigma \subset f}} f\bigg), \qquad
    \esb(\sig) := \relinte \bigg( \bigcup_{\substack{f \in  \Fh^\partial \\ \sigma \cap  f \neq \emptyset}} f \bigg),
\end{equation}
where $\relinte(\cdot)$ denotes the relative interior of a subset of $\Gamma$ with respect to $\Gamma$. 
Similarly, one can think of $\esb(\sigma)$
 as the union of $\stb(v)$ for all the vertices $v\in \sigma$. In particular, if $\sigma$ is a boundary vertex, then $\stb(\sigma)=\esb(\sigma)$, and if $\sigma$ is a boundary face, $\stb(\sigma)=\sigma$.
In this work, we make the following mild assumption. This assumption is the surface counterpart of the assumption made in \cite{EGPV_HO:24} for the extended stars in the bulk and can be satisfied if the mesh is fine enough.

\begin{assumption}[Contractibility of extended boundary stars] \label{ass:cont_esb}
For all $l\in\{0{:}2\}$ and all $\sigma\in\Delta_h^{l,\partial}$, $\clos(\esb(\sigma))$ is contractible. 
\end{assumption}

To provide precise statements on the properties of our discrete-trace preserving commuting projections, we partition the set of mesh cells as $\Th = \Th^{l,\partial} \cup \mTh^l$ with $\mTh^l:=\Th\setminus \Th^{l,\partial}$, for all $l\in\{0{:}3\}$, and
\begin{equation}
    \Th^{l,\partial} := \{\tau\in \Th \::\: \tau\subset\clos(\st(\sigma)), \: \sigma\in \Delta_h^{l,\partial}\}, \quad \forall l\in\{0{:}2\},
\end{equation}
and $\Th^{3,\partial} := \emptyset$, so that $\mTh^3 = \Th$. 
Notice that, for $l\ge1$, tetrahedra in $\mTh^l$ may have nonempty intersection with $\Gamma$, but this intersection has zero $l$-dimensional measure (for instance, $\tau\cap \Gamma$ can contain a boundary vertex or a boundary edge if $\tau \in \mTh^2$, but not a boundary face), whereas tetrahedra in $\Th^{0,\partial}$ (resp., $\mTh^0$) are those that touch (resp., do not touch) the boundary. Notice also that $\Th^{2,\partial} \subset \Th^{1,\partial} \subset \Th^{0,\partial}$. We define the following subsets of $\Omega$:
\begin{equation}
    \Omega^{l,\partial} := \inte \bigg( \bigcup_{\tau\in \Th^{l,\partial}} \tau \bigg).
\end{equation}
We define the $\Omega^{l,\partial}$-extended stars as
\begin{equation} \label{eq:def_es_close_bnd}
    \es^{k,l,\partial}(\tau) := \es^k(\tau)\cap \Omega^{l,\partial}, \quad \forall \tau\in\Th, \quad \forall k\ge1, \quad \forall l\in\{0{:}2\}.
\end{equation}
Another useful partition of the mesh tetrahedra is $\Th = \Th^{l,\partial\partial} \cup \Th^{l,\circ\circ}$, where 
\begin{equation} \label{eq:circcirc}
\Th^{l,\circ\circ}:=\{\tau\in\Th\::\: |\es^{2,l,\partial}(\tau)| = 0\}, \quad
\Th^{l,\partial\partial}:=\Th\setminus \Th^{l,\circ\circ}. 
\end{equation}
In words, for all $\tau \in \Th^{l,\circ\circ}$, there is no mesh cell in $\clos(\es^2(\tau))$ touching the boundary through a subsimplex of dimension larger than $l$.

\subsection{Piecewise polynomial spaces}
\label{sec:pcw_poly}
Let $p\ge0$ be the polynomial degree. For a tetrahedron $\tau \in \Th$, $\pol_p(\tau)$  is the space of polynomials of degree at most $p$ defined on $\tau$,
$\Ne_p(\tau):=\{\bu(\bx) + \bx\times \bv(\bx) : \bu,\bv\in \pol_p(\tau;\R^3)\}$ is the $p$-th order N\'ed\'elec space~\cite{nedelec1980mixed}, and 
$\Rt_p(\tau):=\{\bu(\bx)+v(\bx)\bx: \bu\in \pol_p(\tau;\R^3), v\in\pol_p(\tau)\}$ is the $p$-th order Raviart--Thomas space~\cite{Ra_Tho_MFE_77}. We consider the following piecewise polynomial spaces:
\begin{subequations} \label{eq_spaces_disc} \begin{align}
    V_p^0 &:=  \{ u\in H^1(\Omega): u|_\tau \in \pol_{p+1}(\tau), \forall \tau \in \Th\}, \\
    \bV_p^1 &:=  \{ \bu\in \bH(\curl, \Omega): \bu|_\tau \in \Ne_p(\tau), \forall \tau \in \Th\}, \\
    \bV_p^2 &:= \{ \bu\in \bH(\dive, \Omega): \bu|_\tau \in \Rt_p(\tau), \forall \tau \in \Th\}, \\
    V_p^3 &:= \{ u\in L^2(\Omega): u|_\tau \in \pol_{p}(\tau), \forall \tau \in \Th\}.
\end{align} \end{subequations}
We use the generic notation $V_p^l$ for the above spaces, and note that $V_p^l \subset V^l$, for all $l\in\{0{:}3\}$, with the graph spaces $V^l$ defined in~\eqref{eq:graph_spaces}. 
Moreover, we define the piecewise polynomial spaces having zero trace as
\begin{equation}
\mV_{p}^l:= V_p^l \cap \mV^l.
\end{equation}

\section{Main idea of the construction and main applications}\label{sec:projections}

In this section, we provide the cornerstone result on the construction of bounded, commuting, discrete-trace preserving projections. We also present two useful applications of these projections, providing stable discrete liftings of piecewise polynomial boundary data on the one hand and an optimality result on the discrete versus continuous extension of polynomial data on the other hand. It is useful to define the following continuous spaces with piecewise polynomial trace:
\begin{align}\label{eq:def_contTraceSpaces} 
 \boxed{   V_{\tr,p}^l :=  \{ u\in V^l: \tr^l (u) \in \tr^l(V_{p}^l)\}, \quad \forall l\in\{0{:}3\}.}
\end{align} 
Notice that $V_{\tr,p}^3 = V^3$ for any polynomial degree $p$.

For positive numbers $a,b$, we abbreviate as $a\lesssim b$ the
inequality $a\le Cb$, where the value of the positive constant $C$ can change at each
occurrence, but it can only depend on the shape-regularity parameter $\rho_{\Th}$ of the mesh $\Th$ and the polynomial degree $p$. 

\subsection{The central idea}

In \cite[Section 6]{EGPV_HO:24}, projection operators denoted here $\mPi^l: \mV^l \to \mV_p^l$ for all $l\in\{0{:}3\}$, are constructed such that
\begin{subequations}\label{eq:mPi}
\begin{alignat}{3}
    &\mPi^l (u) = u \quad &&\forall u\in \mV_p^l,\, &&\forall l\in\{0{:}3\}, \label{eq:mPi_projection}\\
    &d^l \mPi^l (u) = \mPi^{l+1}(d^l u) \quad &&\forall u\in \mV^l,\, && \forall l\in\{0{:}2\}, \label{eq:mPi_commuting}\\
    &\|\mPi^l (u)\|_{L^2(\tau)} \lesssim \|u\|_{L^2(\es^2(\tau))} \quad &&\forall u\in \mV^l,\, &&\forall l\in\{0{:}3\},\, \forall \tau\in\Th.\label{eq:mPI_Bound}
\end{alignat}
\end{subequations}
From Section~\ref{sec:construction_Pib} and onward, we will construct operators $\Pi_\partial^l: V^l \to V_p^l$ for all $l\in\{0{:}3\}$ satisfying the following properties:
\begin{subequations}\label{eq:Pib}
\begin{alignat}{4}
&\tr^l(\Pib^l(u)) = \tr^l(u), \quad &&\forall u \in  V_{\tr,p}^l,\,&&  \forall l\in\{0{:}3\},  \label{eq:Pib_trace_preservation}\\
&\tr^{l+1}(d^l \Pib^l(u) ) = \tr^{l+1}(\Pib^{l+1}(d^l u)),  \quad &&\forall u\in V^l,\,&& \forall l\in\{0{:}2\},\label{eq:Pib_commuting}\\
&\Pib^l(u)|_\tau = 0, \quad &&\forall u\in V^l,\,&&\forall l\in\{0{:}2\},\,&&\forall \tau\in\mTh^l, \label{eq:Pib_support}\\
&\|\Pib^l(u)\|_{L^2(\tau)} \lesssim \|u\|_{L^2(\es^{2,l,\partial}(\tau))} + h_\tau \|d^l u \|_{L^2(\es^{2,l,\partial}(\tau))},  \quad &&\forall u\in V^l,\, &&\forall l\in\{0{:}2\},\, &&\forall \tau\in\Th^{l,\partial}, \label{eq:Pib_bound}\\
&\|\Pib^{l}(u)\|_{L^2(\tau)} \lesssim \|u\|_{L^2(\tau)}, \quad &&\forall u\in V^{l},\, &&l=3,\, &&\forall\tau\in\Th.\label{eq:Pib3_bound}
\end{alignat}
\end{subequations}

\begin{theorem}[Central idea]\label{thm:Pi}
Let the operators $\{\mPi^l\}_{l\in\{0{:}3\}}$ satisfy \eqref{eq:mPi} and let the operators $\{\Pib^l\}_{l\in\{0{:}3\}}$ satisfy \eqref{eq:Pib}. Define the linear operators $\Pi^l: V^l \to V_p^l$ as follows:
\begin{equation}\label{eq:Pidef}
\boxed{ \Pi^l:= \mPi^l(I-\Pib^l) + \Pib^l, \quad \forall l\in\{0{:}3\},}
\end{equation}
where $I$ denotes the identity. Then, these operators satisfy
\begin{subequations}
\begin{alignat}{3}
&\Pi^l(u) = u,   \quad &&\forall u \in  V_p^l,\, &&\forall l\in\{0{:}3\}, \label{eq:Pi_projection}\\
&\tr^l(\Pi^l(u)) = \tr^l(u),  \quad &&\forall u \in  V_{\tr,p}^l,\, &&\forall l\in\{0{:}3\},\label{eq:Pi_trace_preservation}\\
&d^l\Pi^l(u) = \Pi^{l+1}(d^l u), \quad &&\forall u \in  V^l,\, &&\forall l\in\{0{:}2\},\label{eq:Pi_commuting}
\end{alignat}
as well as the following bounds: For all $u\in V^l$,
\begin{alignat}{3}
&\|\Pi^l(u)\|_{L^2(\tau)} \lesssim 
\|u\|_{L^2(\es^2(\tau))},\quad 
&& \forall l\in\{0{:}2\},\, && \forall \tau\in \Th^{l,\circ\circ}, \label{eq:Pi_L2bound_m}\\
&\|\Pi^l(u)\|_{L^2(\tau)} \lesssim \|u\|_{L^2(\es^2(\tau)\,\cup\, \es^{4,l,\partial}(\tau))} + h_\tau \|d^l u\|_{L^2(\es^{4,l,\partial}(\tau))},\quad && \forall l\in\{0{:}2\},\, && \forall \tau\in \Th^{l,\partial\partial},  \label{eq:Pi_L2bound_p}\\
&\|d^l \Pi^l(u)\|_{L^2(\tau)}\lesssim \|d^lu\|_{L^2(\es^2(\tau))},\quad
&& \forall l\in\{0{:}2\},\, && \forall \tau\in \Th^{l,\circ\circ},\label{eq:Pi_Hdbound_m} \\
&\|d^l \Pi^l(u)\|_{L^2(\tau)}\lesssim 
\|d^lu\|_{L^2(\es^2(\tau)\,\cup\, \es^{4,l,\partial}(\tau))},\quad
&& \forall l\in\{0{:}2\},\, && \forall \tau\in \Th^{l,\partial\partial}, \label{eq:Pi_Hdbound_p}\\
&\|\Pi^l(u)\|_{L^2(\tau)}\lesssim \|u\|_{L^2(\es^2(\tau))},\quad
&& l=3, \, && \forall \tau\in\Th, \label{eq:Pi_L2bound_3}
\end{alignat}
\end{subequations}
recalling the partition $\Th=\Th^{l,\circ\circ}\cup\Th^{l,\partial\partial}$ introduced in~\eqref{eq:circcirc}.
\end{theorem}

\begin{proof}
We first prove \eqref{eq:Pi_projection}. Let $u \in V_p^l$. Then $u\in V_{\tr,p}^l$, so that, by \eqref{eq:Pib_trace_preservation}, $u- \Pib^l(u) \in \mV_p^l$. Using \eqref{eq:mPi_projection}, we infer that
\begin{equation*}
    \Pi^l(u) = \mPi^l (u- \Pib^l (u)) + \Pib^l(u) = u- \Pib^l (u) + \Pib^l (u) = u.
\end{equation*} 

We next prove \eqref{eq:Pi_trace_preservation}. Let $u\in V_{\tr,p}^l$. Since $\mPi^l$ maps to $\mV_p^l$, we can use \eqref{eq:Pib_trace_preservation} to see that
\begin{equation*}
    \tr^l(\Pi^l(u)) = \tr^l(\mPi^l (u-\Pib^l (u))) + \tr^l(\Pib^l(u)) = \tr^l(\Pib^l(u)) = \tr^l(u).
\end{equation*} 

We now turn to \eqref{eq:Pi_commuting}. Using the commuting property \eqref{eq:mPi_commuting} and again that $u- \Pib^l(u) \in \mV_p^l$ owing to \eqref{eq:Pib_trace_preservation}, we infer that
\begin{alignat*}{2} 
d^l \Pi^l (u) -\Pi^{l+1} (d^l (u)) ={}&d^l \mPi^l\left(u-\Pib^l (u)\right) + d^l\Pib^l (u) \\
&- \mPi^{l+1}\left(d^l u -\Pib^{l+1} (d^l u)\right) - \Pib^{l+1}(d^l u)\\
={}&\mPi^{l+1} (d^l u- d^l \Pib^l(u))+d^l\Pib^l(u) \\
&- \mPi^{l+1} (d^l u-\Pib^{l+1}(d^l u))- \Pib^{l+1} (d^l u)\\
={}& \mPi^{l+1}(\Pib^{l+1} (d^l u)-d^l \Pib^l(u))- (\Pib^{l+1}(d^l u)-d^l \Pib^l(u))
=0.
\end{alignat*}%
In the last step, we used crucially that $\Pib^{l+1}(d^l u)-d^l \Pib^l(u) \in \mV_p^{l+1}$ (which is a consequence of~\eqref{eq:Pib_commuting}) and applied \eqref{eq:mPi_projection}.

Next, we show~\eqref{eq:Pi_L2bound_m}-\eqref{eq:Pi_L2bound_p}. Let $\tau \in \Th$. The triangle inequality followed by~\eqref{eq:mPI_Bound} gives
\begin{align*}
\|\Pi^l(u)\|_{L^2(\tau)} &\le \|\mPi^l(u)\|_{L^2(\tau)} + \|\Pib^l(u)\|_{L^2(\tau)} + \|\mPi^l(\Pib^l(u))\|_{L^2(\tau)}  \\
&\le \|u\|_{L^2(\es^2(\tau))} + \|\Pib^l(u)\|_{L^2(\es^2(\tau))}.
\end{align*}
We use the partition of the mesh tetrahedra introduced in~\eqref{eq:circcirc}. If $\tau \in \Th^{l,\circ\circ}$, then any tetrahedron $\tau'\subset\clos(\es^2(\tau))$ lies outside $\Omega^{l,\partial}$, and therefore is not an element of $\Th^{l,\partial}$. Thus, $\tau'\in \mTh^l$, so that \eqref{eq:Pib_support} implies that $\|\Pib^l(u)\|_{L^2(\es^2(\tau))}=0$. This proves \eqref{eq:Pi_L2bound_m}. If instead $\tau \in \Th^{l,\partial\partial}$, we invoke~\eqref{eq:Pib_bound} and the property 
\[
\bigcup_{\substack{\tau'\in\Th \\ \tau'\subset \clos(\es^{2}(\tau))}} \bigcup_{\substack{\tau'' \in \Th \\ \tau''\subset \clos(\es^{2,l,\partial}(\tau'))}} \tau'' \subseteq \es^{4,l,\partial}(\tau).
\] 
This proves \eqref{eq:Pi_L2bound_p}.
The bounds~\eqref{eq:Pi_Hdbound_m}-\eqref{eq:Pi_Hdbound_p} follow from the commuting property \eqref{eq:Pi_commuting}, the above bounds~\eqref{eq:Pi_L2bound_m}-\eqref{eq:Pi_L2bound_p}, and the complex property $d^{l+1} d^l = 0$.

Finally, \eqref{eq:Pi_L2bound_3} readily follows from \eqref{eq:mPI_Bound} for $l=3$ and \eqref{eq:Pib3_bound}.
\end{proof}

In the following two corollaries, we investigate somewhat sharper stability properties of the discrete-trace preserving projections $\Pi^l$.

\begin{corollary}[$L^2$-stability] \label{cor:L2stab_poly}
Let
$V_{d,p}^l := \{u\in V^l \::\: d^lu\in V^{l+1}_p\}$, for all $l\in\{0{:}2\}$.
Then we have $\|\Pi^l(u)\|_{L^2(\tau)} \lesssim \|u\|_{L^2(\es^2(\tau)\,\cup\, \es^{4,l,\partial}(\tau))}$ for all $u\in V_{d,p}^l$ and all $\tau\in\Th$, and consequently,
\begin{equation} \label{eq:stabL2_d_poly}
\|\Pi^l(u)\|_{L^2(\Omega)} \lesssim \|u\|_{L^2(\Omega)}, \quad \forall u\in V_{d,p}^l, \; \forall l\in\{0{:}2\}.
\end{equation} 
\end{corollary}

\begin{proof}
The bound $\|\Pi^l(u)\|_{L^2(\tau)} \lesssim \|u\|_{L^2(\es^2(\tau)\,\cup\, \es^{4,l,\partial}(\tau))}$ follows from~\eqref{eq:Pi_L2bound_m}-\eqref{eq:Pi_L2bound_p} by invoking an inverse inequality and the regularity of the mesh. Squaring and summing over the mesh cells leads to~\eqref{eq:stabL2_d_poly} owing again to mesh regularity.
\end{proof}

\begin{corollary}[$L^2$-stability with fluctuation]
Let $u\in V^l$ and let $\{\widetilde{\Pi}^l\}_{l\in\{0{:}3\}}$ be the local $L^2$-stable commuting projectors of \cite[Theorem 5.2]{EGPV_HO:24}. Then, 
\begin{equation} \label{eq:stabL2_fluc}
\|\Pi^l(u)\|_{L^2(\Omega)} \lesssim \|u\|_{L^2(\Omega)} + \bigg\{\sum_{\tau\in\Th^{l,\partial}} h_\tau^{2}\|d^l u - \widetilde{\Pi}^{l+1}(d^l u)\|_{L^2(\tau)}^{2} \bigg\}^{\frac12}, \quad \forall u\in V^l, \; \forall l\in\{0{:}2\}.
\end{equation} 
\end{corollary}

\begin{proof}
First, using \eqref{eq:Pi_projection} gives $\Pi^l(u) = \Pi^l(u - \widetilde{\Pi}^l u) + \Pi^l(\widetilde{\Pi}^l u) = \Pi^l(u - \widetilde{\Pi}^l u) + \widetilde{\Pi}^l u$. Hence, using the triangle inequality, we obtain $\|\Pi^l(u)\|_{L^2(\Omega)} \le \|\Pi^l(u - \widetilde{\Pi}^l u)\|_{L^2(\Omega)} + \|\widetilde{\Pi}^l u\|_{L^2(\Omega)}$.
We also immediately have $\|\widetilde{\Pi}^l u\|_{L^2(\Omega)}\lesssim \|u\|_{L^2(\Omega)}$ since $\widetilde{\Pi}^l$ is $L^2$-stable, so that it remains to bound $\|\Pi^l(u - \widetilde{\Pi}^l u)\|_{L^2(\Omega)}$. To this end, we note that we can write $\|\Pi^l(u - \widetilde{\Pi}^l u)\|_{L^2(\Omega)}^2$ as a sum over $\tau\in\Th$ and apply \eqref{eq:Pi_L2bound_m} for all $\tau\in \Th^{l,\circ\circ}$ and \eqref{eq:Pi_L2bound_p} for all $\tau\in\Th^{l,\partial\partial}$. Then invoking the finite overlap of the extended stars (due to mesh shape-regularity), we obtain the global bound
\begin{equation*}
    \|\Pi^l(u - \widetilde{\Pi}^l u)\|_{L^2(\Omega)} \lesssim \|u - \widetilde{\Pi}^l u\|_{L^2(\Omega)} + \bigg\{ \sum_{\tau \in \Th^{l,\partial\partial}} h_\tau^{2} \|d^l(u -\widetilde{\Pi}^l u)\|_{L^2(\es^{4,l,\partial}(\tau))}^{2} \bigg\}^{\frac12}.
\end{equation*}
We bound the first term on the right-hand side by invoking the triangle inequality and the $L^2$-stability of $\widetilde{\Pi}^l$. For the second term, we invoke the commuting property of $\widetilde{\Pi}^l$, together with mesh regularity and the fact that any $\tau'\in\clos(\es^{4,l,\partial}(\tau))$ is, by definition, in $\Th^{l,\partial}$. This proves \eqref{eq:stabL2_fluc}.
\end{proof}

\subsection{Application to stable discrete liftings}\label{sec:disc_lifting}
In this section, we give two useful applications of the commuting discrete-trace preserving projections constructed in Theorem \ref{thm:Pi}. 

We first show that these projections lead to stable discrete liftings of boundary data.
For all $l\in\{0{:}2\}$, we equip the trace space $\tr^l(V^l)$ with the following norm: For all $g \in \tr^l(V^l)$, 
\begin{alignat}{1}\label{eq:trace_norm_def}
    \|g\|_{\tr^l(V^l)}:= \inf_{\substack{v \in V^l \\ \tr^l(v)=g}} \|v\|_{V^l}.
\end{alignat}
It is well-known that, for all $l\in\{0{:}2\}$, the harmonic extension operator $E^l:\tr^l(V^l)\to V^l$ provides stable liftings of boundary data in $\tr^l(V^l)$. Indeed,  we have, for all $g\in \tr^l(V^l)$, 
\begin{subequations}\label{eq:harmonic_extension_props}
    \begin{align}
        &\tr^l(E^l(g)) = g, \label{eq:harmonic_extension_trace}\\ 
        &\|E^l(g)\|_{V^l} \lesssim \|g\|_{\tr^l(V^l)}. \label{eq:harmonic_extension_bound}
    \end{align}
\end{subequations}
We now show that we can combine such lifting operators with the projections constructed in Theorem \ref{thm:Pi} to obtain discrete stable liftings. 
\begin{lemma}[Discrete stable lifting operators]\label{lem:disc_stab_lif}
Let $\Pi^l: V^l \rightarrow V_p^l$ be the projection operators constructed in Theorem \ref{thm:Pi} for all $l\in\{0{:}2\}$. Let $E^l : \tr^l(V^l)\to V^l$ be the harmonic extension operators. Define the linear operators $L_p^l: \tr^l(V_p^l) \rightarrow V_p^l$ as $L_p^l = \Pi^l \circ E^l$ for all $l\in\{0{:}2\}$. Then, for all $g_h \in \tr^l(V_p^l)$,
\begin{subequations} 
\begin{alignat}{2}
\tr^l( L_p^l (g_h)) &= g_h, \quad &&   \label{trace111}  \\
\|L_p^l(g_h)\|_{V^l} &\lesssim \| g_h\|_{\tr^l(V^l)}. \quad  && \label{trace112}
\end{alignat}
\end{subequations}
\end{lemma}

\begin{proof}
Let $g_h\in \tr^l(V_p^l)$ and set $u:=E^l(g_h)$. Owing to \eqref{eq:harmonic_extension_trace}, we infer that $\tr^l(u)=g_h$. In particular,
$u\in V^l_{\tr,p}$, and so  by \eqref{eq:Pi_trace_preservation}, 
\begin{equation*}
    \tr^l(L_p^l(g_h))=\tr^l(\Pi^l(u))=\tr^l(u)=g_h,
\end{equation*}
which proves \eqref{trace111}.

To show \eqref{trace112}, we first note that summing \eqref{eq:Pi_L2bound_m}-\eqref{eq:Pi_L2bound_p} and \eqref{eq:Pi_Hdbound_m}-\eqref{eq:Pi_Hdbound_p} over all $\tau\in\Th$ and invoking the finite overlap of extended stars (which follows from the shape regularity of the mesh) leads to the global estimate $\|\Pi^l(u)\|_{V^l}\lesssim \|u\|_{V^l}$. Then, invoking \eqref{eq:harmonic_extension_bound} proves \eqref{trace112} and completes the proof. 
\end{proof}

In many boundary-value problems, one seeks solutions $u \in V^l$ with $\tr^l(u)= g$ on $\partial \Omega$ where $g\in\tr^l(V^l)$ is some given boundary data. Very often one approximates $g$ by a discrete object $g_h\in \tr^l(V_p^l)$ and one finds an approximation $u_h \in V_p^l$ such that $\tr^l(u_h) = g_h$.  Classical arguments in error analysis lead to the bound
\begin{equation}
\|u-u_h\|_{V^l} \lesssim \inf_{\substack{v_h \in V_p^l \\ \tr^l(v_h) = g_h}} \|u-v_h\|_{V^l}. 
\end{equation}
A natural question is then to estimate the right-hand side in terms of the best-approximation of $u$ over the whole space $V^l_p$. As shown in \cite{Ains_Guzm_Saya_16}, one can obtain such an estimate as long as one has a discrete lifting. 

\begin{corollary}[Best-approximation with and without boundary data]
We have, for all $l\in\{0{:}2\}$,
\begin{equation}\label{eq:bestapproxresult}
\inf_{\substack{v_h \in V_p^l \\ \tr^l(v_h) = g_h}} \| u-v_h\|_{V^l} \lesssim \inf_{v_h^* \in V_p^l} \| u-v_h^*\|_{V^l} + \|g-g_h\|_{\tr^l(V^l)}.     
\end{equation}
\end{corollary}

\begin{proof}
We invoke the discrete lifting operators $L_p^l$ constructed in Lemma \ref{lem:disc_stab_lif}. Let $v_h^* \in V_p^l$ and set $w_h := v_h^*-L_p^l( \tr^l(v_h^*)-g_h)$. We see that $w_h\in V_p^l$, $\tr^l(w_h)=g_h$ due to \eqref{trace111}, and
\begin{equation*}
 \| u-w_h\|_{V^l}=\|u-v_h^*- L_p^l(\tr^l(v_h^*)-g_h) \|_{V^l} \lesssim \| u-v_h^*\|_{V^l} + \| \tr^l(v_h^*) -g_h\|_{\tr^l(V^l)},
\end{equation*}
where we used the triangle inequality and \eqref{trace112}. Moreover, using the triangle inequality and the definition of the norm \eqref{eq:trace_norm_def}, we obtain
\begin{alignat*}{1}
    \| \tr^l(v_h^*) - g_h\|_{\tr^l(V^l)} & \le \| g-g_h\|_{\tr^l(V^l)}+  \| \tr^l(u-v_h^*)\|_{\tr^l(V^l)} \\
    &\le \| g-g_h\|_{\tr^l(V^l)}+  \| u-v_h^*\|_{V^l}.
\end{alignat*}
Combining these two inequalities, we obtain 
\begin{equation*}
    \inf_{\substack{v_h \in V_p^l \\ \tr^l(v_h) = g_h}} \| u-v_h\|_{V^l} \le \|u-w_h\|_{V^l}\lesssim \|u-v_h^*\|_{V^l}+\|g-g_h\|_{\tr^l(V^l)}.
\end{equation*} 
Taking the infimum over $v_h^* \in V^l_p$ concludes the proof.
\end{proof}

We close this section with another application of the operators $\{\Pi^l\}_{l\in\{0{:}2\}}$ in showing the equivalence of the discrete and continuous $L^2$-minimizers of certain polynomial data.

\begin{corollary}[Extension of polynomial data]
For all $l\in\{0{:}2\}$, let $g^l_p \in \tr^l(V^l_p)$ and let $f^{l+1}_p \in d^lV^{l}_p$.
The following holds:
\begin{equation} \label{eq:min_min_poly}
\min_{\substack{v_h^l \in V^l_p \\
\tr^l(v_h^l) = g^l_p,\;
d^lv_h^l = f^{l+1}_p}} \|v_h^l\|_{L^2(\Omega)} \lesssim \min_{\substack{v^l \in V^l \\
\tr^l(v^l) = g^l_p,\;
d^lv^l = f^{l+1}_p}} \|v^l\|_{L^2(\Omega)}.
\end{equation}
\end{corollary}

\begin{proof}
Let $\mu_h$ and $\mu$ be the values attained by the two minimization problems in~\eqref{eq:min_min_poly}. We need to prove that $\mu_h\lesssim \mu$.
(Notice in passing that the converse bound $\mu\le \mu_h$ is trivial.)
Let $v^l\in V^l$ be such that $\tr^l(v^l) = g^l_p$ and $d^lv^l = f^{l+1}_p$. 
Set $v_h^l:= \Pi^l(v) \in V^l_p$. Since $v^l\in V_{\tr,p}^l$, we infer from~\eqref{eq:Pi_trace_preservation} that $\tr^l(v_h^l) = g^l_p$. Moreover, \eqref{eq:Pi_commuting} and~\eqref{eq:Pi_projection} imply that $d^lv_h^l = f^{l+1}_p$. Hence, $v_h^l$ is in the discrete minimization set, so that $\mu_h \lesssim \| \Pi^l(v^l)\|_{L^2(\Omega)}$. Since $v^l\in V_{d,p}^l$, as defined in Corollary~\ref{cor:L2stab_poly}, the bound~\eqref{eq:stabL2_d_poly} therein implies that $\| \Pi^l(v^l)\|_{L^2(\Omega)} \lesssim \|v^l\|_{L^2(\Omega)}$. Thus, $\mu_h \lesssim \|v^l\|_{L^2(\Omega)}$ for all $v^l$ from the continuous minimization set. Hence, $\mu_h \lesssim \mu$, and the proof is complete. 
\end{proof}

\section{Construction of the operators $\Pib^l$}
\label{sec:construction_Pib}

In this section, we present the main ideas underlying the construction of the operators $\Pib^l$ for all $l\in\{0{:}2\}$. The construction of the operator $\Pib^3$ is simpler and is detailed separately.

\subsection{Decomposition into lowest-order and higher-order components}

For all $l\in\{0{:}2\}$, we consider a decomposition of the trace space $\tr^l(V^l_p)$ into lowest-order and higher-order components of the form
\begin{subequations} \label{eq:linear_dec} \begin{equation} \label{eq:direct_dec}
\tr^l(V^l_p) = \tr^l(V^l_{0}) \oplus \tr^l(V^l_{+,p}),
\end{equation}
where the supplementary space $V^l_{+,p}$ will be described in more detail 
in Section~\ref{sec:high_order} (the space $V^l_{+,p}$ is nontrivial only if $p\ge1$). 
For the time being, we just introduce two families of 
functions $\{B_{0,r}^l\}_{r\in I_0^l}\subset V^l_{p}$ and $\{B_{+,r}^l\}_{r\in I_+^l}\subset V_{p}^l$, with corresponding index sets $I_0^l$ and $I_+^l$, so that
\begin{equation} \label{eq:tr_Vl_span}
\tr^l(V_{0}^l) = \Span_{r\in I_0^l} (\tr^l(B_{0,r}^l)), \quad \tr^l(V_{+,p}^l) = \Span_{r\in I_+^l} (\tr^l(B_{+,r}^l)).
\end{equation} \end{subequations}
As further detailed in Section~\ref{sec:lowest_order}, the functions $\{B_{0,r}^l\}_{r\in I_0^l}$ are merely the lowest-order Whitney forms. However, to stay general, we keep the notation $B_{0,r}^l$ in what follows.
The direct sum~\eqref{eq:direct_dec} induces the following trace spaces:
\begin{subequations} \begin{align}
V_{\tr,0}^l &:= \{u \in V^l : \tr^l(u) \in \tr^l(V^l_{0})\}, \\
V_{\tr,+}^l &:= \{u \in V^l : \tr^l(u) \in \tr^l(V^l_{+,p})\}.
\end{align} \end{subequations}
Although not specifically required at this stage, we will see that the families
$\{\tr^l(B_{0,r}^l)\}_{r\in I_0^l}$ and $\{\tr^l(B_{+,r}^l)\}_{r\in I_+^l}$ are linearly
independent. This follows from the properties~\eqref{eq:Kron_zeta} and~\eqref{eq:Kron_zeta+} below.

\subsection{The linear maps $\alpha_{0,r}^l$ and $\alpha_{+,r}^l$}
\label{sec:alpha}

Our construction hinges on linear maps from $V^l$ to $\R$ which are denoted
$(\alpha_{0,r}^l)_{r\in I_0^l}$ and $(\alpha_{+,r}^l)_{r\in I_+^l}$, for all $l\in\{0{:}2\}$. 
Since both families of maps are constructed using the same principles, we drop here the subscripts $0$ and $+$ and use a unified presentation for linear maps which we generically denote as $(\alpha_{r}^l)_{r\in I^l}$.

The construction is based on piecewise polynomial functions defined on the boundary and having suitable properties as detailed in Section~\ref{sec:bnd_weights} below. These functions are called boundary weights. Specifically, we consider the following families of boundary weights:
\begin{equation}\label{eq:zeta_space}
\{\zeta_r^0\}_{r\in I^0} \subset L^2(\Gamma), \quad \{\bzeta_r^1\}_{r\in I^1} \subset \bH(\sdive, \Gamma), \quad \{\zeta_r^2\}_{r\in I^2} \subset H^1(\Gamma).
\end{equation}
From these boundary weights, we define piecewise polynomial extensions $(Y_r^l)_{r\in I^l}$ 
defined on $\Omega$ such that
\begin{equation}\label{eq:trace_Y_zeta}
\tr^2(\bY_r^0) = \zeta_r^0, \quad \tr^1_\perp(\bY_r^1) = \bzeta_r^1, \quad \tr^0(Y_r^2) = \zeta_r^2.
\end{equation}
Finally, the linear maps $\alpha_r^l:V^l\to \R$ are defined so that
\begin{equation}\label{eq:alpha_def}
\boxed{
\begin{alignedat}{3}
\alpha_r^0(w) &:= (w, \dive(\bY_r^0))_{\Omega} + (\grad(w), \bY_r^0)_{\Omega}, \quad &&\forall r\in I^0,\; &&\forall w\in V^0,\\
\alpha_r^1(\bw) &:= (\bw, \curl(\bY_r^1))_{\Omega} - (\curl(\bw), \bY_r^1)_{\Omega}, \quad &&\forall r\in I^1,\; &&\forall \bw\in \bV^1,\\
\alpha_r^2(\bw) &:= (\bw, \grad(Y_r^2))_{\Omega} + (\dive(\bw), Y_r^2)_{\Omega}, \quad &&\forall r\in I^2,\; &&\forall \bw\in \bV^2.
\end{alignedat}}
\end{equation}

\begin{lemma}[Properties of the linear maps $\alpha_r^l$] \label{lem:pties_alpha}
The following holds for all $r\in I^l$:
\begin{subequations} \label{eq:alpha_bnd} 
\begin{alignat}{2} 
&\alpha_r^l(w) = (\tr^l(w), \zeta_r^l)_{\Gamma}, \quad &&\forall w\in C^1(\overline{\Omega}) \cup V^l_p, \label{eq:alpha_bnd_C1}\\
&\alpha_r^l(w) = 0, \quad &&\forall w\in \mV^l, \label{eq:alpha_bnd_mV} \\
&\alpha_r^l(w) = (\tr^l(w), \zeta_r^l)_{\Gamma}, \quad &&\forall w\in V_{\tr,p}^l,\label{eq:alpha_bnd_tr}
\end{alignat} 
\end{subequations}
where the right-hand sides in \eqref{eq:alpha_bnd} are $L^2(\Gamma)$- or $\bL^2(\Gamma)$-inner products.
\end{lemma}

\begin{proof}
\eqref{eq:alpha_bnd_C1} follows by invoking \eqref{eq:zeta_space} and \eqref{eq:trace_Y_zeta} and integration by parts. \eqref{eq:alpha_bnd_mV} follows from considering in~\eqref{eq:alpha_bnd_C1} a smooth, compactly supported function $w$ and invoking a density argument. Finally, to prove~\eqref{eq:alpha_bnd_tr}, we consider $w\in V_{\tr,p}^l$ so that there is $v\in V_p^l$ so that $\tr^l(w) =\tr^l(v)$. Since $w-v\in \mV^l$, \eqref{eq:alpha_bnd_mV} implies that $\alpha_r^l(w)=\alpha_r^l(v)$. Moreover, since $v\in V_p^l$, \eqref{eq:alpha_bnd_C1} gives $\alpha_r^l(v) = (\tr^l(v), \zeta_r^l)_{\Gamma}$. Since $\tr^l(v)=\tr^l(w)$, we conclude that~\eqref{eq:alpha_bnd_tr} holds true.
\end{proof}

\subsection{Two-level construction of the operators $\Pib^l$ for all $l\in\{0{:}2\}$}

We define the operators $P_0^l:V^l\to V^l_{0}$ and $P_+^l:V^l\to V^l_{+,p}$ such that, for all $u\in V^l$ and all $l\in\{0{:}2\}$,
\begin{equation} \label{eq:def_P0+}
\boxed{P_0^l(u) := \sum_{r\in I_0^l} \alpha_{0,r}^l(u) B_{0,r}^l, \qquad
P_+^l(u) := \sum_{r\in I_+^l} \alpha_{+,r}^l(u) B_{+,r}^l.}
\end{equation}
Using these operators, we define
\begin{equation}\label{eq:def_Pib_HO}
\boxed{\Pib^l: = P_0^l + P_+^l (I - P_0^l), \quad \forall l\in\{0{:}2\}.}
\end{equation}
The goal of this section is to identify sufficient conditions on the operators $\{P_0^l\}_{l\in\{0{:}2\}}$ and $\{P_+^l\}_{l\in\{0{:}2\}}$ so that the operators $\{\Pi_\partial^l\}_{l\in\{0{:}2\}}$ satisfy \eqref{eq:Pib}.

\begin{remark}[$p=0$] \label{rem:p=0}
When $p=0$, we simply set $P_+^l\equiv0$ so that $\Pib^l=P_0^l$.
\end{remark}

We assume that the following properties hold true:
\begin{subequations}\label{eq:Pb_properties}
\begin{alignat}{4}
&\tr^l(P_0^l(u))=\tr^l(u), \quad &&\forall u \in V_{\tr,0}^l,\, && \forall l\in\{0{:}2\},\, \label{eq:P_trace_preservation}\\
&\tr^{l+1} (P_0^{l+1}(d^l u)) = \tr^{l+1} (d^l P_0^l (u)), \quad &&\forall u\in V^l,\, && \forall l\in\{0{:}1\},\label{eq:P_trace_commuting}\\
&P_0^l(u)|_\tau = 0, \quad &&\forall u\in V^l,\, &&\forall l\in\{0{:}2\},\, &&\forall \tau\in \mTh^l,\label{eq:P_support}\\
&\|P_0^l(u)\|_{L^2(\tau)} \lesssim \|u\|_{L^2(\es^{1,l,\partial}(\tau))} + h_\tau \|d^l u\|_{L^2(\es^{1,l,\partial}(\tau))}, \quad &&\forall u\in V^l,\, && \forall l\in\{0{:}2\},\, &&\forall \tau\in\Th^{l,\partial}. \label{eq:P_bound}
\end{alignat}
\end{subequations}
We notice that the assumptions~\eqref{eq:Pb_properties} only concern the operators $\{P_0^l\}_{l\in\{0{:}2\}}$.
Moreover, we also assume that the following properties hold true:
\begin{subequations}\label{eq:Qb_properties}
\begin{alignat}{4}
&\tr^l(P_+^l(u))=\tr^l(u), \quad 
\tr^l(P_0^l(u))=0 \quad &&\forall u \in V_{\tr,+}^l,\, && \forall l\in\{0{:}2\},\label{eq:Q_trace_preservation_0}\\
&\tr^{l+1} (P_+^{l+1} (d^l u)) = \tr^{l+1} (d^l P_+^l (u)), \quad &&\forall u\in V^l,\, && \forall l\in\{0{:}1\},\label{eq:Q_trace_commuting}\\
&P_+^l(u)|_\tau = 0, \quad &&\forall u\in V^l,\, &&\forall l\in\{0{:}2\},\, &&\forall \tau\in \mTh^l, \label{eq:Q_support}\\
&\|P_+^l(u)\|_{L^2(\tau)} \lesssim \|u\|_{L^2(\es^{1,l,\partial}(\tau))} + h_\tau \|d^l u\|_{L^2(\es^{1,l,\partial}(\tau))}, \quad &&\forall u\in V^l,\, &&\forall l\in\{0{:}2\}, \, &&\forall\tau\in\Th^{l,\partial}.\label{eq:Q_bound}
\end{alignat}
\end{subequations}
We notice that the assumptions~\eqref{eq:Qb_properties} concern the operators $\{P_+^l\}_{l\in\{0{:}2\}}$, except for the second statement in~\eqref{eq:Q_trace_preservation_0} which concerns the operators $\{P_0^l\}_{l\in\{0{:}2\}}$. We will identify in Section~\ref{sec:bnd_weights} sufficient conditions on the basis functions $\{B_{0,r}^l\}_{r\in I_0^l}$ and $\{B_{+,r}^l\}_{r\in I_+^l}$ and on the boundary weights $(\zeta_{0,r}^l)_{r\in I_0^l}$ and $(\zeta_{+,r}^l)_{r\in I_+^l}$ so that the above assumptions~\eqref{eq:Pb_properties} and~\eqref{eq:Qb_properties} are indeed fulfilled (see Lemmas~\ref{lem:cond_zeta0} and~\ref{lem:cond_zeta+}, respectively).

\begin{theorem}[Satisfaction of~\eqref{eq:Pib} for all $l\in\{0{:}2\}$]\label{thm:Satisfaction_of_Pib}
Assume that the properties~\eqref{eq:Pb_properties} and~\eqref{eq:Qb_properties} hold true. Then, the operators $\{\Pi_\partial^l\}_{l\in\{0{:}2\}}$ defined in~\eqref{eq:def_Pib_HO} satisfy \eqref{eq:Pib_trace_preservation}-\eqref{eq:Pib_bound}.
\end{theorem}

\begin{proof}
Proof of \eqref{eq:Pib_trace_preservation}. We need to prove that $\tr^l(\Pib^l(u)) = \tr^l(u)$ for all $u\in V_{\tr,p}^l$ and all $l\in\{0{:}2\}$. 
There is $v\in V^l_p$ so that $\tr^l(u)=\tr^l(v)$ and, owing to~\eqref{eq:direct_dec}, we have $\tr^l(v)=\tr^l(v_0)+\tr^l(v_+)$ with $v_0\in V^l_{0}$ and $v_+\in V^l_{+,p}$. Set $u_0:=v_0$ and $u_+:=u-v_0$ so that $u=u_0+u_+$. Observe that $u_0\in V^l_{\tr,0}$ since $\tr^l(u_0)=\tr^l(v_0)$ and $v_0\in V^l_{0}$. Similarly, $u_+\in V^l_{\tr,+}$ since $\tr^l(u_+)=\tr^l(u-v_0)=\tr^l(v-v_0)=\tr^l(v_+)$ and $v_+\in V^l_{+,p}$. When $p=0$, we have $\tr^l(v_+) = \tr^l(u_+) \equiv 0$. Let us prove that $\tr^l((I-\Pib^l)(u_0))=0$ and $\tr^l((I-\Pib^l)(u_+))=0$. Owing to~\eqref{eq:P_trace_preservation}, $\tr^l(u_0-P_0^l(u_0))=0$. This implies that $u_0-P_0^l(u_0) \in V^l_{\tr,+}$ as well, so that $\tr^l((I-P_+^l)(u_0-P_0^l(u_0)))=0$ by the first statement in~\eqref{eq:Q_trace_preservation_0}. Hence, $\tr^l((I-\Pib^l)(u_0))=0$. Moreover, the second statement in~\eqref{eq:Q_trace_preservation_0} implies that $\tr^l(P_0^l(u_+))=0$. Hence,
$\tr^l((I-\Pib^l)(u_+))=\tr^l((I-P_+^l)(u_+))=0$, where the second equality follows from the
first statement in~\eqref{eq:Q_trace_preservation_0}. This completes the proof of~\eqref{eq:Pib_trace_preservation}. 

Proof of \eqref{eq:Pib_commuting}. We need to prove that $\tr^{l+1}(d^l \Pib^l(u) ) = \tr^{l+1}(\Pib^{l+1}(d^l u))$ for all $u\in V^l$ and all $l\in\{0{:}1\}$. We compute using \eqref{eq:P_trace_commuting} followed by \eqref{eq:Q_trace_commuting},
\begin{align*}
\tr^{l+1} &\big\{ d^l \Pi_\partial^l (u) - \Pi_\partial^{l+1} (d^l u) \big\} \\
&= \tr^{l+1} \big\{ d^l P_0^l (u) + d^l P_+^l (u - P_0^l (u)) 
- P_0^{l+1} (d^l u) -P_+^{l+1} (d^l u - P_0^{l+1} (d^l u)) \big\}\\
&= \tr^{l+1} \big\{ d^l P_+^l (u) - d^l P_+^l (P_0^l (u)) - P_+^{l+1}(d^l u) + P_+^{l+1} (P_0^{l+1} (d^l u))\big\} \\
&= \tr^{l+1} \big\{ P_+^{l+1} (P_0^{l+1} (d^l u) - d^l P_0^l (u))\big\}.
\end{align*}
Since $\tr^{l+1}\big\{ P_0^{l+1} (d^l u) - d^l P_0^l (u)\big\} = 0$, we have $P_0^{l+1} (d^l u) - d^l P_0^l (u) \in V_{\tr,+}^{l+1}$. Hence, using the
first statement in~\eqref{eq:Q_trace_preservation_0}, this gives
\begin{equation*}
\tr^{l+1} (P_+^{l+1} (P_0^{l+1} (d^l u) - d^l P_0^l (u))) = \tr^{l+1} (P_0^{l+1} (d^l u) - d^l P_0^l (u)) = 0.
\end{equation*}
This proves~\eqref{eq:Pib_commuting}.

Proof of \eqref{eq:Pib_support}. The fact that $\Pi_{\partial}^l(u)|_{\tau}=0$ for all $u\in V^l$, all $l\in\{0{:}2\}$, and all $\tau\in\mTh^l$ readily follows from \eqref{eq:P_support} and \eqref{eq:Q_support}.

Proof of \eqref{eq:Pib_bound}. We need to prove that $\|\Pib^l(u)\|_{L^2(\tau)} \lesssim \|u\|_{L^2(\es^{2,l,\partial}(\tau))} + h_\tau \|d^l u \|_{L^2(\es^{2,l,\partial}(\tau))}$ for all $u\in V^l$, all $l\in\{0{:}2\}$, and all $\tau \in \Th^{l,\partial}$.
Using the triangle inequality, \eqref{eq:P_bound}, and \eqref{eq:Q_bound} leads to
\begin{align*}
\|\Pib^l (u)\|_{L^2(\tau)} &\leq \|P_+^l (u- P_0^l(u))\|_{L^2(\tau)} + \|P_0^l(u)\|_{L^2(\tau)} \\
&\lesssim \|u- P_0^l(u)\|_{L^2(\es^{1,l,\partial}(\tau))} + h_\tau \|d^l (u- P_0^l(u))\|_{L^2(\es^{1,l,\partial}(\tau))} + \|P_0^l (u)\|_{L^2(\tau)} \\
&\lesssim \|u\|_{L^2(\es^{1,l,\partial}(\tau))} + h_\tau\|d^l u\|_{L^2(\es^{1,l,\partial}(\tau))} + \|P_0^l (u)\|_{L^2(\es^{1,l,\partial}(\tau))} + h_\tau\|d^l P_0^l(u)\|_{L^2(\es^{1,l,\partial}(\tau))} \\
&\lesssim \|u\|_{L^2(\es^{1,l,\partial}(\tau))} + h_\tau\|d^l u\|_{L^2(\es^{1,l,\partial}(\tau))} + \|P_0^l (u)\|_{L^2(\es^{1,l,\partial}(\tau))},
\end{align*}
where we used an inverse inequality and the mesh regularity to bound $\|d^l P_0^l(u)\|_{L^2(\es^{1,l,\partial}(\tau))}$. Finally, to estimate $\|P_0^l (u)\|_{L^2(\es^{1,l,\partial}(\tau))}$, we use \eqref{eq:P_bound} and the fact that $\es^{1,l,\partial}(\tau') \subseteq \es^{2,l,\partial}(\tau)$ for all $\tau'\in\Th$ such that $\tau'\subset \clos(\es^{1,l,\partial}(\tau))$.
\end{proof}

\subsection{Construction of the operator $\Pib^3$}\label{sec:Pib3_construction}
Since $\tr^3$ corresponds to integration over $\Omega$ and $V_{\tr,p}^3 = V^3$ for all $p\ge0$, we can take $\Pi_\partial^3$ to be the ``canonical projection'' onto $V_0^3$, i.e., we set $\Pi_\partial^3:V^3 \to V_0^3$ so that
\begin{equation} \label{eq:def_Pi3}
\Pi_\partial^3 (u) := \sum_{\tau\in\Th} \alpha_\tau^3(u) B^3_\tau,
\end{equation}
where $\alpha_\tau^3(u):=\int_\tau u$, $B^3_\tau:=|\tau|^{-1}\chi_\tau$, and $\chi_\tau$ is the characteristic function of $\tau$ for all $\tau\in\Th$. We make the following assumptions on the operators $P_0^2$ and $P_+^2$:
\begin{subequations}\label{eq:l=2_properties}
\begin{alignat}{3}
    &\tr^{3} (\dive (P_0^{2}(\bu))) = \tr^{3}(\dive(\bu)), \quad && \forall \bu\in \bV^{2},\label{eq:P_trace_l=2}\\ 
    &\tr^{3}(\dive(P_+^2(\bu))) = 0, \quad && \forall \bu\in \bV^2.\label{eq:Q_trace_l=2}
\end{alignat}
\end{subequations}
We will identify in Section \ref{sec:bnd_weights} sufficient conditions on the lowest-order boundary weights $(\zeta_{0,r}^l)_{r\in I_0^l}$ and the higher-order basis functions $\{B_{+,r}^l\}_{r\in I_+^l}$ so that the assumptions~\eqref{eq:l=2_properties} are indeed fulfilled (see Lemmas~\ref{lem:cond_zeta0} and~\ref{lem:cond_zeta+}, respectively).

\begin{lemma}[Satisfaction of~\eqref{eq:Pib} for $l=3$]\label{lem:Pi3_partial} 
Let $\Pib^3$ be defined in~\eqref{eq:def_Pi3}. Then, the statements in~\eqref{eq:Pib} for $l=3$ hold true. 
\end{lemma}

\begin{proof}
Proof of \eqref{eq:Pib_trace_preservation} for $l=3$. We need to show that 
$\tr^3(\Pi_\partial^3 (u)) = \tr^3(u)$ for all $u\in V^3$. Since $B^3_\tau$ is supported in $\tau$ and $\int_\Omega B^3_\tau = \int_\tau B^3_\tau=1$, the expected identity follows from
\begin{equation*}
\int_\Omega \Pi_\partial^3 (u) = \sum_{\tau\in\Th} \alpha^3_\tau(u) \int_\Omega B^3_\tau = \sum_{\tau\in\Th} \int_\tau u  = \int_\Omega u.
\end{equation*}
    
Proof of \eqref{eq:Pib_commuting} for $l=2$. We need to show that $\tr^3(\dive(\Pi_\partial^2(\bu))) = \tr^3(\Pi_\partial^3 (\dive(\bu)))$ for all $\bu \in \bV^2$. By \eqref{eq:Pib_trace_preservation} for $l=3$, since $\dive(\bu)\in V^3$, it suffices to show that $\tr^3(\dive(\Pi_\partial^2(\bu))) = \tr^3(\dive(\bu))$. This is indeed the case since, owing to~\eqref{eq:l=2_properties}, we have
\begin{equation*}
    \tr^3(\dive(\Pi_\partial^2(\bu))) = \tr^3(\dive(P_0^2(\bu))) + \tr^3(\dive(P_+^2(\bu-P_0^2(\bu)))) = \tr^3(\dive(\bu)).
\end{equation*}

Proof of \eqref{eq:Pib_bound} for $l=3$. We need to show that $\|\Pi_\partial^3(u)\|_{L^2(\tau)} \lesssim \|u\|_{L^2(\tau)}$ for all $u\in V^3$ and all $\tau\in\Th$. 
Since $\Pi_\partial^3 (u)|_\tau = \alpha_\tau^3(u)B^3_\tau$, the assertion readily follows from the Cauchy--Schwarz inequality which gives $|\alpha^3_\tau(u)| \le |\tau|^\frac{1}{2} \|u\|_{L^2(\tau)}$ and $\|B_\tau^3\|_{L^2(\tau)}\le |\tau|^{-\frac12}$.
\end{proof}

\subsection{Conditions on the basis functions and on the boundary weights}
\label{sec:bnd_weights}

The goal of this section is to identify sufficient conditions on the basis functions $\{B_{0,r}^l\}_{r\in I_0^l}$ and $\{B_{+,r}^l\}_{r\in I_+^l}$ and on the boundary weights $(\zeta_{0,r}^l)_{r\in I_0^l}$ and $(\zeta_{+,r}^l)_{r\in I_+^l}$, for all $l\in\{0{:}2\}$, so that the resulting operators $\{P^l_0\}_{l\in\{0{:}2\}}$ and $\{P^l_+\}_{l\in\{0{:}2\}}$ defined in~\eqref{eq:def_P0+} satisfy the expected properties~\eqref{eq:Pb_properties} and~\eqref{eq:Qb_properties}, respectively, as well as \eqref{eq:l=2_properties}.
We will present concrete examples fulfilling these assumptions in Sections~\ref{sec:lowest_order} and~\ref{sec:high_order}, respectively. We first detail the lowest-order case, and then deal with the higher-order case.

\subsubsection{Lowest-order case}

We make the following assumptions on the basis functions $\{B_{0,r}^l\}_{r\in I_0^l}$ and the boundary weights $(\zeta_{0,r}^l)_{r\in I_0^l}$.

\begin{assumption}[Lowest-order basis functions]
The basis functions $\{B_{0,r}^l\}_{r\in I_0^l}$ satisfy the following two properties.
\textup{(i)} Relation to differential operators: For all $l\in\{0{:}1\}$, there exist ``derivative coefficients'' $(\kappa_0^{r',r})_{(r',r)\in I_0^{l+1}\times I_0^l}$ so that 
\begin{subequations} \label{eq:ass_B0}
\begin{equation}
\tr^{l+1} \big( d^l B_{0,r}^l \big) = \sum_{r'\in I_0^{l+1}}\kappa_0^{r',r} \tr^{l+1}(B_{0,r'}^{l+1}),
\quad\forall r\in I_0^l. \label{eq:d_varphi_relation0}
\end{equation}
\textup{(ii)} Support and norm: For each $l\in\{0{:}2\}$ and $r\in I_0^l$, there exists a boundary geometric object $\sigma_{0,r}^l \in \Delta_h^{l,\partial}$ and a real number $\beta_{0,r}^l$ so that
\begin{equation} \label{eq:support_B0}
\supp(B_{0,r}^l) \subset \clos(\st(\sigma_{0,r}^l)), \qquad
\|B_{0,r}^l\|_{L^2(\Omega)} \lesssim h_{\sigma_{0,r}^l}^{\beta_{0,r}^l}.
\end{equation}
\end{subequations}
\end{assumption}

\begin{assumption}[Lowest-order boundary weights]
The boundary weights $(\zeta_{0,r}^l)_{r\in I_0^l}$ satisfy the following properties: 
\textup{(i)} Relation to basis functions: For all $l\in\{0{:}2\}$,
\begin{subequations} \label{eq:ass_zeta0} \begin{equation} \label{eq:Kron_zeta}
(\zeta_{0,r}^l,\tr^l(B_{0,r'}^l))_\Gamma = \delta_{r,r'}, \quad \forall r,r'\in I_0^l.
\end{equation}
\textup{(ii)} Relation to derivative coefficients and surface differential operators:
\begin{alignat}{2}
-\sdive (\bzeta_{0,r'}^1) &= \sum_{r \in I_0^0} \kappa_0^{r',r} \zeta_{0,r}^0, \quad &&\forall r'\in I_0^1, \label{deltazeta1}\\
\srot (\zeta_{0,r'}^2) &= \sum_{r\in I_0^1} \kappa_0^{r',r} \bzeta_{0,r}^1, \quad &&\forall r'\in I_0^2.\label{deltazeta2} 
\end{alignat}
\textup{(iii)} Support and boundedness: For all $l\in\{0{:}2\}$ and all $r\in I_0^l$,
recalling the extension $Y_{0,r}^l$ defined in~\eqref{eq:trace_Y_zeta}, we have 
\begin{equation} \label{eq:support_Y0}
\supp(Y_{0,r}^l) \subseteq \clos(\es(\sigma_{0,r}^l)), \qquad
h_{\sigma_{0,r}^l} \|d^{2-l} Y_{0,r}^l\|_{L^2(\Omega)} \lesssim \|Y_{0,r}^l\|_{L^2(\Omega)}  \lesssim h_{\sigma_{0,r}^l}^{1-\beta_{0,r}^l}.
\end{equation}
\end{subequations}
\textup{(iv)} Partition of unity on $\Gamma$ when $l=2$:
\begin{equation}\label{eq:l=2pou}
    \sum_{r\in I_0^2} \tr^3(\dive(\bB_{0,r}^2))\, \zeta_{0,r}^2 \equiv 1\:\: \text{on } \Gamma.
\end{equation}
\end{assumption}

\begin{lemma}[Fulfillment of~\eqref{eq:Pb_properties}] \label{lem:cond_zeta0}
Assume~\eqref{eq:linear_dec}, \eqref{eq:ass_B0} and~\eqref{eq:ass_zeta0}. 
Then, the operators $\{P_0^l\}_{l\in\{0{:}2\}}$ defined in~\eqref{eq:def_P0+} satisfy the expected properties~\eqref{eq:Pb_properties} and \eqref{eq:P_trace_l=2}.
\end{lemma}

\begin{proof}
Proof of~\eqref{eq:P_trace_preservation}. We need to prove that 
$\tr^l(P_0^l(u))=\tr^l(u)$ for all $u \in V_{\tr,0}^l$ and all $l\in\{0{:}2\}$.
We observe that $\tr^l(u-P_0^l(u)) \in \tr^l(V_0^l)$ by assumption on $u$ and by 
construction of $P_0^l$. Therefore, recalling~\eqref{eq:tr_Vl_span} and~\eqref{eq:Kron_zeta}, to show that $\tr^l(u-P_0^l(u))$ vanishes, it suffices
to show that 
\begin{equation} \label{eq:proof_5.9.a}
(\zeta_{0,r}^l,\tr^l(u-P_0^l(u)))_\Gamma=0 \quad \forall r\in I_0^l.
\end{equation} 
By the definition
of $P_0^l$ and invoking~\eqref{eq:Kron_zeta}, we infer that 
$(\zeta_{0,r}^l,\tr^l(P_0^l(u)))_\Gamma = \alpha_{0,r}^l(u)$. Moreover, $(\zeta_{0,r}^l,\tr^l(u))_\Gamma = \alpha_{0,r}^l(u)$ owing to~\eqref{eq:alpha_bnd_tr} since $u \in V_{\tr,0}^l \subset V_{\tr,p}^l$. This proves~\eqref{eq:proof_5.9.a}.

Proof of~\eqref{eq:P_trace_commuting}. We need to prove that 
$\tr^{l+1} (P_0^{l+1}(d^l u)) = \tr^{l+1} (d^l P_0^l (u))$ for all $u\in V^l$ and all $l\in\{0{:}1\}$.
We first compute using \eqref{eq:d_varphi_relation0},
\begin{equation*}
\tr^{l+1}(d^l P_0^l(u)) = \sum_{r\in I_0^l} \alpha_{0,r}^l(u) \tr^{l+1} (d^l B_{0,r}^l) = \sum_{r\in I_0^l }\sum_{r'\in I_0^{l+1}} \alpha_{0,r}^l(u) \kappa_0^{r',r} \tr^{l+1} (B_{0,r'}^{l+1}).
\end{equation*}
On the other hand, using Lemma~\ref{lem:alpha_der} below (which uses~\eqref{deltazeta1}-\eqref{deltazeta2}), we infer that
\begin{equation*}
\tr^{l+1} (P_0^{l+1}(d^l u)) = \sum_{r'\in I_0^{l+1}}\alpha_{0,r'}^{l+1}(d^l u) \tr^{l+1}(B_{0,r'}^{l+1}) = \sum_{r'\in I_0^{l+1}}\sum_{r\in I_0^l} \kappa_0^{r',r}\alpha_{0,r}^l(u) \tr^{l+1}(B_{0,r'}^{l+1}).
\end{equation*}
Swapping the order of summations proves the expected commuting identity.

Proof of~\eqref{eq:P_support}. We need to show that $P_{0}^l(u)|_\tau = 0$ for all $u\in V^l$, all $l\in\{0{:}2\}$, and all $\tau\in\mTh^l$. By definition of $\mTh^l$, $\tau$ does not lie in $\clos(\st(\sigma_{0, r}^l))$ for any $r\in I_0^l$. Therefore, by \eqref{eq:support_B0}, $B_{0,r}^l|_\tau = 0$ for all $r\in I_0^l$, which implies \eqref{eq:P_support}.

Proof of~\eqref{eq:P_bound}. We need to prove that 
$\|P_0^l(u)\|_{L^2(\tau)} \lesssim \|u\|_{L^2(\es^{1,l,\partial}(\tau))} + h_\tau \|d^l u\|_{L^2(\es^{1,l,\partial}(\tau))}$ for all $u\in V^l$, all $l\in\{0{:}2\}$, and all $\tau\in\Th^{l,\partial}$. Let us set $I_{0}^l(\tau):=\{r\in I_0^l\::\: \tau \subset \clos(\st(\sigma_{0,r}^l))\}$ and $\omega_{0,r}^l:=\es(\sigma_{0,r}^l)$. The definition~\eqref{eq:def_P0+} of $P_0^l$, together with the triangle inequality and~\eqref{eq:support_B0} gives
\begin{align*}
\|P_0^l(u)\|_{L^2(\tau)} &\le \sum_{r\in I_0^l(\tau)} |\alpha_{0,r}^l(u)| \, \|B_{0,r}^l\|_{L^2(\tau)}
 \\
&\le \sum_{r\in I_{0}^l(\tau)} \Big\{ \|u\|_{L^2(\omega_{0,r}^l)} \|d^{2-l}Y_{0,r}^l\|_{L^2(\Omega)} +
\|d^lu\|_{L^2(\omega_{0,r}^l)} \|Y_{0,r}^l\|_{L^2(\Omega)} \Big\} \|B_{0,r}^l\|_{L^2(\tau)},
\end{align*}
where we used the definition~\eqref{eq:alpha_def} of the linear maps $\alpha_{0,r}^l$,
the assumption~\eqref{eq:support_Y0} to localize the norms on $u$ and $d^lu$,
and the Cauchy--Schwarz inequality.
Invoking the bounds from~\eqref{eq:support_B0} and~\eqref{eq:support_Y0} together with the shape-regularity of the mesh, we obtain
\[
\|P_0^l(u)\|_{L^2(\tau)} \lesssim \sum_{r\in I_{0}^l(\tau)}
\Big\{ \|u\|_{L^2(\omega_{0,r}^l)} + h_\tau \|d^lu\|_{L^2(\omega_{0,r}^l)} \Big\}.
\]
To conclude the proof of \eqref{eq:P_bound}, we observe that any mesh cell $\tau'\subset \clos(\omega_{0,r}^l)$ for some $r\in I_0^l(\tau)$ is in $\clos(\es^{1,l,\partial}(\tau))$ (recall that $\es^{1,l,\partial}(\tau):=\es(\tau)\cap \Omega^{l,\partial}$, see~\eqref{eq:def_es_close_bnd}). Indeed, $\tau'\subset \clos(\es(\tau))$ since $\sigma_{0,r}^l\in \tau \cap \tau'$, so that $\tau \cap \tau' \ne \emptyset$. Moreover, $\tau'\in \Th^{l,\partial}$ since $\sigma_{0,r}^l \in \Delta_h^{l,\partial}$. Hence, $\tau' \subset \clos(\esb^{1,l,\partial})$. 

Proof of \eqref{eq:P_trace_l=2}. We need to prove that $\tr^3(\dive(P_0^2(\bu))) = \tr^3(\dive(\bu))$ for all $\bu\in\bV^2$. Setting $Y^2 := \sum_{r\in I_0^2} \tr^3(\dive(\bB_{0,r}^2))Y_{0,r}^2 \in V^0$, \eqref{eq:trace_Y_zeta} and \eqref{eq:l=2pou} imply that $\tr^0(Y^2) = 1$ on $\Gamma$. A similar reasoning to the proof of \eqref{eq:alpha_bnd_mV} then proves that $(\bu, \grad(Y^2))_{\Omega} + (\dive(\bu),Y^2)_{\Omega} = \tr^3(\dive(\bu))$. This gives
\begin{align*}
    \tr^3(\dive(\bu)) &= \sum_{r\in I_0^2} \left((\bu, \grad(Y_{0,r}^2))_{\Omega} + (\dive(\bu), Y_{0,r}^2)_{\Omega}\right)\tr^3(\dive(\bB_{0,r}^2))\\
    &= \sum_{r\in I_0^2} \alpha_{0,r}^2(\bu) \tr^3(\dive(\bB_{0,r}^2)) = \tr^3(\dive(P_0^2(\bu))).
\end{align*}
This concludes the proof.
\end{proof}

\begin{lemma}[Action of $\alpha_{0,r}^l$ on derivatives] \label{lem:alpha_der}
Assume~\eqref{deltazeta1}-\eqref{deltazeta2}.
The following holds for all $u\in V^l$, all $l\in\{0{:}1\}$, and all $r'\in I_0^{l+1}$:
\begin{equation}\label{eq:alpha_commuting}
\alpha_{0,r'}^{l+1}(d^l u) = \sum_{r\in I_0^l} \kappa_0^{r',r} \alpha_{0,r}^l(u).
\end{equation}
\end{lemma}

\begin{proof}
Proof for $l=0$. Consider first $u\in C^1(\overline{\Omega})$. We obtain, for all $r'\in I_0^1$,
\begin{alignat*}{2}
\alpha_{0,r'}^1(\grad(u)) &= (\tr^1(\grad(u)), \bzeta_{0,r'}^1)_{\Gamma} \quad &&\text{by \eqref{eq:alpha_bnd_C1}}\\
&=(\sgrad(\tr^0(u)), \bzeta_{0,r'}^1)_{\Gamma} \quad &&\text{by \eqref{eq:trace_d_commuting}}\\
&= ( \tr^0(u), -\sdive(\bzeta_{0,r'}^1))_{\Gamma} \quad &&\text{by \eqref{eq:sdive_adjoint}}\\
&= \sum_{r \in I_{0,r}^0} \kappa_0^{r',r} ( \tr^0(u),  \zeta_{0,r}^0)_{\Gamma} \quad &&\text{by \eqref{deltazeta1}}\\
&= \sum_{r\in I_{0,r}^0} \kappa_0^{r',r} \alpha_{0,r}^0(u) \quad &&\text{by \eqref{eq:alpha_bnd_C1}}.
\end{alignat*}
By density, we conclude that~\eqref{eq:alpha_commuting} holds true for $l=0$. 

Proof for $l=1$. Consider first $\bu\in \bC^1(\overline{\Omega})$. We obtain, for all $r'\in I_0^2$,
\begin{alignat*}{2}
\alpha_{0,r'}^2(\curl(\bu)) &= (\tr^2 (\curl(\bu)), \zeta_{0,r'}^2)_{\Gamma} \quad &&\text{by \eqref{eq:alpha_bnd_C1}}\\
&= (\scurl(\tr^1(\bu)), \zeta_{0,r'}^2)_{\Gamma} \quad &&\text{by \eqref{eq:trace_d_commuting}}\\
&= ( \tr^1(\bu), \srot(\zeta_{0,r'}^2))_{\Gamma} \quad &&\text{by \eqref{eq:scurl_adjoint}}\\
&= \sum_{r \in I_{0,r}^1} \kappa_0^{r',r} ( \tr^1(\bu),  \bzeta_{0,r}^1)_{\Gamma} \quad &&\text{by \eqref{deltazeta2}}\\
&= \sum_{r\in I_{0,r}^1} \kappa_0^{r',r} \alpha_{0,r}^1(\bu) \quad &&\text{by \eqref{eq:alpha_bnd_C1}}.
\end{alignat*}
By density, we conclude that~\eqref{eq:alpha_commuting} holds true for $l=1$. 
\end{proof}

\subsubsection{Higher-order case}

The assumptions in the higher-order case are essentially the same as in the lowest-order case, so that we only sketch the main arguments. There is, however, one additional assumption on the basis functions (see~\eqref{eq:mean_zero_div_l=2}) and one on the boundary weights (see~\eqref{eq:Kron_zeta0+}). 

\begin{assumption}[Higher-order basis functions] \label{ass:B+}
The basis functions $\{B_{+,r}^l\}_{r\in I_+^l}$ satisfy the following three properties.
\textup{(i)} Relation to differential operators: For all $l\in\{0{:}1\}$, there exist ``derivative coefficients'' $(\kappa_+^{r',r})_{(r',r)\in I_+^{l+1}\times I_+^l}$ so that 
\begin{subequations} \label{eq:ass_B+}
\begin{equation}
\tr^{l+1} \big( d^l B_{+,r}^l \big) = \sum_{r'\in I_+^{l+1}}\kappa_+^{r',r} \tr^{l+1}(B_{+,r'}^{l+1}).
\quad\forall r\in I_+^l. \label{eq:d_varphi_relation+}
\end{equation}
\textup{(ii)} Support and norm: For all $l\in\{0{:}2\}$ and all $r\in I_+^l$, there exists a boundary geometric object $\sigma_{+,r}^l \in \Delta_h^{\partial}$ and a real number $\beta_{+,r}^l$ so that
\begin{equation} \label{eq:support_B+}
\supp(B_{+,r}^l) \subset \clos(\st(\sigma_{+,r}^l)), \qquad 
\|B_{+,r}^l\|_{L^2(\Omega)} \lesssim h_{\sigma_{+,r}^l}^{\beta_{+,r}^l}.
\end{equation}
\textup{(iii) Mean-zero divergence when $l=2$:
\begin{equation}\label{eq:mean_zero_div_l=2}
    \tr^3(\dive(\bB_{+,r}^2)) = 0, \quad \forall r\in I_{+}^2.
\end{equation}
}
\end{subequations}
\end{assumption}

\begin{assumption}[Higher-order boundary weights] \label{ass:zeta+}
The boundary weights $(\zeta_{+,r}^l)_{r\in I_+^l}$ satisfy the following properties: 
\textup{(i)} Relation to basis functions: For all $l\in\{0{:}2\}$,
\begin{subequations} \label{eq:ass_zeta+}
\begin{alignat}{2} 
&(\zeta_{+,r}^l,\tr^l(B_{+,r'}^l))_\Gamma = \delta_{r,r'}, &\quad &\forall r,r'\in I_+^l, \label{eq:Kron_zeta+} \\
&(\zeta_{0,r}^l,\tr^l(B_{+,r'}^l))_\Gamma = 0, &\quad &\forall (r,r')\in I_0^l\times I_+^l. \label{eq:Kron_zeta0+}
\end{alignat}
\textup{(ii)} Relation to derivative coefficients and surface differential operators:
\begin{alignat}{2}
-\sdive (\bzeta_{+,r'}^1) &= \sum_{r \in I_+^0} \kappa_+^{r',r} \zeta_{+,r}^0, \quad &&\forall r'\in I_+^1, \label{deltazeta1+}\\
\srot (\zeta_{+,r'}^2) &= \sum_{r\in I_+^1} \kappa_+^{r',r} \bzeta_{+,r}^1, \quad &&\forall r'\in I_+^2.\label{deltazeta2+} 
\end{alignat} 
\textup{(iii)} Support and boundedness: For all $l\in\{0{:}2\}$ and all $r\in I_+^l$, recalling the extension $Y_{+,r}^l$ defined in~\eqref{eq:trace_Y_zeta}, we have
\begin{equation}\label{eq:Ybounds+}
\supp(Y_{+,r}^l) \subseteq \clos(\st(\sigma_{+,r}^l)), \quad 
h_{\sigma_{+,r}^l} \|d^{2-l} Y_{+,r}^l\|_{L^2(\Omega)} \lesssim \|Y_{+,r}^l\|_{L^2(\Omega)}  \lesssim h_{\sigma_{+,r}^l}^{1-\beta_{+,r}^l}.
\end{equation}
\end{subequations}
\end{assumption}

\begin{lemma}[Fulfillment of~\eqref{eq:Qb_properties}] \label{lem:cond_zeta+}
Assume~\eqref{eq:linear_dec}, \eqref{eq:ass_B+} and~\eqref{eq:ass_zeta+}. 
Then, the operators $\{P_+^l\}_{l\in\{0{:}2\}}$ defined in~\eqref{eq:def_P0+} satisfy the expected properties~\eqref{eq:Qb_properties} and \eqref{eq:Q_trace_l=2}.
\end{lemma}

\begin{proof}
The only novelty in showing the properties~\eqref{eq:Qb_properties} with respect to the lowest-order case is to prove that $\tr^l(P_0^l(u))=0$ for all $u \in V_{\tr,+}^l$ and all $l\in\{0{:}2\}$. By definition, we have $\tr^l(P_0^l(u))=\sum_{r\in I_0^l} \alpha_{0,r}^l(u) \tr^l(B_{0,r}^l)$. Since $u \in V_{\tr,+}^l$, we infer from~\eqref{eq:tr_Vl_span} that there are real numbers $(\gamma_r(u))_{r\in I_+^l}$ so that $\tr^l(u) = \sum_{r\in I_+^l} \gamma_r(u) B_{+,r}^l$. Invoking~\eqref{eq:Kron_zeta0+} gives $(\zeta_{0,r}^l,\tr^l(u))_\Gamma = 0$ for all $r\in I_0^l$. But, since $u\in  V_{\tr,+}^l\subset  V_{\tr,p}^l$, \eqref{eq:alpha_bnd_tr} implies that $\alpha_{0,r}^l(u) = (\zeta_{0,r}^l,\tr^l(u))_\Gamma$ for all $r\in I_0^l$. Hence, $\tr^l(P_0^l(u))=0$. Moreover, the property \eqref{eq:Q_trace_l=2} follows immediately from \eqref{eq:mean_zero_div_l=2}.
\end{proof}

\begin{remark}[Bound~\eqref{eq:Q_bound}]
Since $\supp(Y_{+,r}^l) \subseteq \clos(\st(\sigma_{+,r}^l))$, the support of the high-order extensions is a bit more compact than the one of the lowest-order extensions (which involves $\es(\sigma_{0,r}^l)$, see~\eqref{eq:support_Y0}). Thus, the bound~\eqref{eq:Q_bound} can be made tighter by invoking norms of $u$ and $d^lu$ on subsets of $\es^{1,l,\partial}(\tau)$. For simplicity, we adopted a uniform presentation for both $P_0^l$ and $P_+^l$.
\end{remark}

\section{Construction of the lowest-order basis functions and boundary weights} 
\label{sec:lowest_order}

In this section, we detail the construction of the lowest-order basis functions and boundary weights. We remark that the basis functions are indeed of lowest order, but the corresponding boundary weights are of higher order even in the lowest-order case.

\subsection{Preliminaries}
\label{sec:prelim}

his section contains preliminary results to prepare for the construction of the lowest-order basis functions and boundary weights. We introduce the canonical degrees of freedom and the corresponding Whitney forms (in the bulk), piecewise polynomial spaces on the boundary mesh, and the Alfeld split of the boundary mesh.

\subsubsection{Canonical degrees of freedom and Whitney forms in the bulk}
\label{sec:Whitney}

The canonical degrees of freedom are linear forms on $V_p^l$, for all $l\in\{0{:}3\}$,
associated with the (oriented) $l$-simplices of the mesh $\Th$ and which are defined as follows:
\begin{subequations}\label{eq:Whitney_dofs}\begin{alignat}{2}
\phi_v(u)&:=u(v),&\quad&\forall u\in V_p^0,\; \forall v \in \Delta_h^0=\Vh,\\
\phi_e(\bu)&:=\int_e \bu\SCAL\bt_e,&\quad&\forall \bu\in \bV_p^1,\; \forall e \in \Delta_h^1=\Eh,\\
\phi_f(\bu)&:=\int_f \bu\SCAL\bn_f,&\quad&\forall \bu\in \bV_p^2,\; \forall f \in \Delta_h^2=\Fh,\\
\phi_\tau(u)&:=\int_\tau u,&\quad&\forall u\in V_p^3,\; \forall \tau \in \Delta_h^3=\Th.
\end{alignat}\end{subequations}
The above integrals are understood in algebraic form; for instance,
$\int_\tau 1 := |\tau| \Sign(\det(\bt_{e_1},\bt_{e_2},\bt_{e_3}))$ with $\bt_{e_k}$ pointing from $x_{i_0}$ to $x_{i_k}$ for all $k\in\{1{:}3\}$.
It is well-known that $\{\phi_v\}_{v\in\Vh}$, $\{\phi_e\}_{e\in\Eh}$, $\{\phi_f\}_{f\in\Fh}$, and $\{\phi_\tau\}_{\tau\in\Th}$ form a basis for the dual space of the lowest-order piecewise polynomial spaces $V_0^0$, $\bV_0^1$, $\bV_0^2$, and $V_0^3$, respectively. 

The dual bases in $\{V^l_0\}_{l\in\{0{:}3\}}$ of the canonical degrees of freedom are composed of the so-called Whitney forms whose scalar and vector-valued proxies we denote by $\{W_v\}_{v\in\Vh}$, $\{\bW_e\}_{e\in\Eh}$, $\{\bW_f\}_{f\in\Fh}$, $\{W_\tau\}_{\tau\in\Th}$. By construction, these are such that
\begin{equation*}
V_0^0 = \Span \{W_v\}_{v\in\Vh}, \quad \bV_0^1 = \Span \{\bW_e\}_{e\in\Eh}, \quad \bV_0^2 = \Span \{\bW_f\}_{f\in\Fh}, \quad V_0^3 = \Span \{W_\tau\}_{\tau\in\Th},
\end{equation*}
and
\begin{equation} \label{eq:dof_Whitney}
\phi_{v'}(W_v)=\delta_{vv'}, \quad
\phi_{e'}(\bW_e)=\delta_{ee'}, \quad
\phi_{f'}(\bW_f)=\delta_{ff'}, \quad
\phi_{\tau'}(W_\tau)=\delta_{\tau\tau'},
\end{equation}
where the $\delta$'s are Kronecker deltas. We introduce the generic notation $W_\sigma^l$ 
for all $l\in\{0{:}3\}$ and all $\sigma\in\Delta_h^l$. We have (see, e.g., \cite{BBF13, ErnGuermondbook})
\begin{equation}\label{CW}
\supp(W_\sigma^l) = \clos(\st(\sigma)), \qquad
\|W_\sigma^l\|_{L^2(\st(\sigma))} \le C_W h_{\sigma}^{\frac32-l}, 
\end{equation}
where $C_W$ only depends on the shape-regularity parameter $\rho_{\Th}$ of the mesh $\Th$.

Incidence matrices are the algebraic realization of the 
differential operators from the de Rham complex~\eqref{complex}, 
but acting on the lowest-order degrees of freedom. Their entries are incidence numbers in $\{-1,0,1\}$ associated with 
pairs of oriented geometric objects. To define these
numbers, we introduce the subsets
\begin{subequations} \begin{alignat}{2}
\Ev&:=\{e\in\Eh\,:\, v\in e\},
&&\quad \forall v\in\Vh, \\
\Fe&:=\{f\in\Fh\,:\, e\subset f\}, &&\quad \forall e\in\Eh, \\
\Tf&:=\{\tau\in\Th\,:\, f\subset \tau\}, &&\quad 
\forall f\in\Fh.
\end{alignat} \end{subequations}
If $e:=[x_{i_0},x_{i_1}] \in \Ev$, $v$ is obtained from $e$ by omitting one of the two vertices
of $e$, say $x_{i_j}$ with $j\in\{0,1\}$, and we set $\iota_{ev}:=(-1)^j$. 
If $f:=[x_{i_0},x_{i_1},x_{i_2}] \in \Fe$, $e$ is obtained from $f$ by omitting one of the three vertices of $f$, say $x_{i_j}$ with $j\in\{0,1,2\}$, and we set $\iota_{fe}:=(-1)^j$. 
If $\tau=[x_{i_0},x_{i_1},x_{i_2},x_{i_3}] \in \Tf$, $f$ is obtained from $\tau$ by omitting one of the four vertices of $\tau$, say $x_{i_j}$ with $j\in\{0,1,2,3\}$, and we set $\iota_{\tau f}:=(-1)^j$.
Finally, we also set $\iota_{ev}:=0$ for all $v\in \Vh$ and all $e\not\in\Ev$,
$\iota_{fe}:=0$ for all $e\in \Eh$ and all $f\not\in\Fe$, and
$\iota_{\tau f}:=0$ for all $f\in \Fh$ and all $\tau\not\in\Tf$.

The key property of the incidence numbers is the following algebraic realization of the complex property of \eqref{complex}:
\begin{subequations} \label{eq:compatible_iota} \begin{alignat}{2}
\sum_{v\in\Ve} \iota_{ev} &=0, \quad &&\forall e\in\Eh, \label{eq:GRAD} \\
\sum_{e\in\Ev\cap \Ef} \iota_{fe}\iota_{ev} &=0, \quad && \forall f\in \Fh, \; \forall v\in \Vh, \label{eq:CURL_GRAD} \\
\sum_{f\in\Fe \cap \Ft} \iota_{\tau f}\iota_{fe} &=0, \quad && \forall \tau\in \Th, \; \forall e\in \Eh, \label{eq:DIV_CURL}
\end{alignat} \end{subequations}
where $\Ve:=\{v\in\Vh\,:\, v\in e\}$ for all $e\in \Eh$,
$\Vf:=\{v\in\Vh\,:\, v\in f\}$ and $\Ef:=\{e\in\Eh\,:\, e\subset f\}$ for all $f\in \Fh$,
and $\Vt:=\{v\in\Vh\,:\, v\in \tau\}$, $\Et :=\{e\in\Eh\,:\, e\subset \tau\}$, and $\Ft :=\{f\in\Fh\,:\, f\subset \tau\}$ for all $\tau\in \Th$.

The following identities are a straightforward consequence of the Stokes theorem~\cite{whitney2012geometric}:
\begin{equation} \label{eq:W_iota}
\grad \, (W_v) = \sum_{e\in \Ev} \iota_{ev} \bW_e,
\quad
\curl \, (\bW_e) = \sum_{f\in \Fe} \iota_{fe} \bW_f,
\quad
\dive (\bW_f) = \sum_{\tau\in \Tf} \iota_{\tau f} W_\tau.
\end{equation}

\subsubsection{Discrete trace spaces}\label{sec:boundary_discrete_spaces}



For all $l\in \{0{:}2\}$ and all $\sigma\in \Delta_h^{l,\sigma}$, we set $W^{l,\partial}_\sigma:=\tr^l(W^l_\sigma)$, so that $\tr^l(V_0^l) = \Span \{W_\sigma^{l,\partial}\}_{\sigma\in\Delta_h^{l,\partial}}$. Moreover, we have
\begin{equation}\label{eq:CWB}
\supp(W_\sigma^{l,\partial}) = \clos(\stb(\sigma)), \qquad
\|W_\sigma^{l, \partial}\|_{L^2(\stb(\sigma))} \le C_{W^\partial} h_{\sigma}^{1-l}, 
\end{equation}
where $C_{W^\partial}$ only depends on the shape-regularity parameter $\rho_{\Th}$ of the mesh $\Th$.
Owing to \eqref{eq:W_iota}, \eqref{eq:trace_d_commuting}, and the fact that $\tr^l(W_\sigma^{l,\partial}) = 0$ for all $\sigma\in\mathring{\Delta}_h^l$ and all $l\in\{0{:}2\}$, we have
\begin{subequations}\label{WB_iota}
\begin{alignat}{2}
    \sgrad (W_v^\partial) &= \sum_{e\in \Ev \cap \Eh^\partial} \iota_{ev} \bW_e^\partial \quad &&\forall v\in\Vh^\partial,\\
    \scurl (\bW_e^\partial) &= \sum_{f\in \Fe \cap \Fh^\partial} \iota_{fe} W_f^\partial \quad &&\forall e\in\Eh^\partial.
\end{alignat}
\end{subequations}

We also introduce the following functionals on the boundary spaces $\{\tr^l(V_p^l)\}_{l\in\{0{:}2\}}$:
\begin{subequations}\begin{alignat}{2}
\phi_v^\partial(u)&:=u(v),&\quad&\forall u\in \tr^0(V_p^0),\; \forall v \in \Delta_h^{0,\partial}=\Vh^\partial,\\
\phi_e^\partial(\bu)&:=\int_e \bu\SCAL\bt_e,&\quad&\forall \bu\in \tr^1(\bV_p^1), \; \forall e \in \Delta_h^{1,\partial}=\Eh^\partial,\\
\phi_f^\partial(u)&:=\int_f u,&\quad&\forall u\in \tr^2(V_p^2),\; \forall f \in \Delta_h^{2,\partial}=\Fh^\partial.
\end{alignat}\end{subequations}
We notice that $\phi^\partial_\sigma \circ \tr^l = \phi_\sigma$ for all $\sigma \in \Delta_h^{l,\partial}$ and all $l\in\{0{:}2\}$, where the $\phi_\sigma$'s are the canonical dofs defined in~\eqref{eq:Whitney_dofs}. Moreover, we readily see that
\begin{subequations} \begin{alignat}{3}
&\sum_{v\in\Ve} \iota_{ve} \phi^\partial_v(u) = \phi_e^\partial(\sgrad u),\quad &&\forall u\in \tr^0(V^0_p),\;&&\forall e\in \Eh^\partial, \label{eq:iota_ev_bnd}\\
&\sum_{e\in\Ef} \iota_{ef} \phi^\partial_e(\bu) = \phi_f^\partial(\scurl \bu),\quad &&\forall \bu\in \tr^1(\bV^1_p),\;&&\forall f\in \Fh^\partial.\label{eq:iota_fe_bnd}
\end{alignat} \end{subequations}

\subsubsection{Boundary mesh: Alfeld split and useful properties}

As in \cite{EGPV_HO:24}, we use solutions to local problems on mesh splits to construct the local weight functions. However, the mesh split we employ in this work is an Alfeld split on the boundary faces. More precisely, this boundary Alfeld split is obtained by adding the barycenter to every boundary face $f\in\Fh^\partial$ and connecting it to the three vertices of $f$. This produces three sub-triangles of $f$. We denote the collection of all boundary faces obtained from this split as $\FhBA$. 

In general, the restriction of the Alfeld split of the bulk mesh is not the boundary Alfeld split. Hence, we will need another bulk mesh split in order to extend discrete functions defined on the boundary Alfeld split to the interior. One such mesh split is the Worsey--Farin split \cite{WorseyFarin87}. In particular, the Worsey--Farin split of $\Th$ is obtained by first adding an interior vertex to every mesh cell $\tau\in\Th$ and connecting it to the four vertices of $\tau$, and then for each face $f\in\Fh$, one connects the barycenter of $f$ to the three vertices of $f$ and to the added interior vertices of the adjacent tetrahedra. We will in particular use the barycenter as the added interior vertex for every $\tau\in\Th$. The resulting mesh is denoted $\ThWF$, and the discrete spaces $V_p^l$ over this mesh are denoted generically as $V_p^l(\ThWF)$ for all $l\in\{0{:}3\}$.

For all $l\in\{0{:}2\}$ and all $\sigma\in\Delta_h^{l,\partial}$, recall the definition~\eqref{eq:def_esb}  of the extended boundary star of $\sigma$, $\esb(\sigma)$. We let $\calFesb$ be the collection of boundary faces in $\Fh^\partial$ composing $\esb(\sigma)$. We define the local boundary spaces $V_p^l(\calFesb)$ as 
\begin{equation}\label{eq:local_spaces}
    V_p^l(\calFesb) := \{v|_{\esb(\sig)} : v=\tr^l(u),\; u\in V_p^l\}.
\end{equation}
Similarly, letting $\calFesbA:=\{ F\in \FhBA\::\: F\subset \clos(\esb(\sigma))\}$, we define the spaces $\{V_p^l(\calFesbA)\}_{l\in\{0{:}2\}}$ by replacing $\calFesb$ with $\calFesbA$ and $V_p^l$ by $V_p^l(\ThWF)$ in \eqref{eq:local_spaces}. Then, we define
\begin{subequations}\label{eq:local_boundary_spaces_m}
\begin{align}
\mV_p^0(\calFesbA) &:= \left\{u\in V_p^0(\calFesbA) :  u|_{\partial \esb(\sigma)} = 0\right\},\\
\mbV_{p}^1(\calFesbA) &:= \left\{\bu\in \bV_{p}^1(\calFesbA) : \bu|_e \cdot \bt_e = 0,\, \forall e\in \Eh^\partial, e\subset \partial\esb(\sigma)\right\},\label{eq:def_mV_pp^1}\\
\mV_p^2(\calFesbA) &:= \left\{u\in V_p^2(\calFesbA) : (u, 1)_{\esb(\sigma)} = 0\right\}.
\end{align}
\end{subequations}
We also set 
\begin{equation}
    \mbV_{p,\perp}^1(\calFesbA):= \bn \times \mbV_{p}^1(\calFesbA).
\end{equation}
While these spaces are defined on local boundary patches, we also identify them with spaces defined over all of $\Gamma$ by their natural zero-extensions to the rest of $\Gamma$.
Owing to Assumption~\ref{ass:cont_esb}, the following discrete sequences are exact:
\begin{subequations} 
\begin{equation}\label{localcomplex}
    \mathbb{R}
    \stackrel{\subset}{\xrightarrow{\hspace*{0.5cm}}}\
     V_p^0(\calFesb)
    \stackrel{\sgrad}{\xrightarrow{\hspace*{0.5cm}}}\
      \bV_p^1(\calFesb)
    \stackrel{\scurl}{\xrightarrow{\hspace*{0.5cm}}}\
      \bV_p^2(\calFesb)
    \stackrel{0}{\xrightarrow{\hspace*{0.5cm}}}\
    0,
\end{equation}
\begin{equation}\label{localcomplexm}
    0 
    \stackrel{\subset}{\xrightarrow{\hspace*{0.5cm}}}\
     \mVp{0}(\calFesbA)
    \stackrel{\srot}{\xrightarrow{\hspace*{0.5cm}}}\
      \mbVpp{1} (\calFesbA)
    \stackrel{\sdive}{\xrightarrow{\hspace*{0.5cm}}}\
      \mbVp{2}(\calFesbA)
    \stackrel{\int_{\esb(\sigma)}}{\xrightarrow{\hspace*{0.5cm}}}\
    0.
\end{equation}
\end{subequations}
We define the following kernels in the local spaces:
\begin{subequations} \label{eq:def_ZZ} \begin{align}
\Zz   V_p^0(\calFesb)&:=  \{ u \in  V_p^0(\calFesb): \sgrad \, u=\bzero \},  \\
\Zz   \bV_p^1(\calFesb)&:= \{ \bu \in  \bV_p^1(\calFesb): \scurl \, \bu=0 \}, 
\end{align} \end{subequations}
and the orthogonal complements as
\begin{subequations} \label{eq:def_ZZ_perp} \begin{align}
\Zz^{\perp}   V_p^0(\calFesb)&:= \{ u \in  V_p^0(\calFesb): (u,v)_{\esb(\sigma)} =0, \forall v \in  \Zz   V_p^0(\calFesb)\}, \\ 
\Zz^{\perp}   \bV_p^1(\calFesb)&:= \{ \bu \in  \bV_p^1(\calFesb): (\bu, \bv)_{\esb(\sigma)} = 0, \forall \bv \in  \Zz   \bV_p^1(\calFesb) \}.
\end{align} \end{subequations}
We define the spaces  $\Zz \mVp{0}(\calFesbA)$, $\Zz \mbV_{p,\perp}^1(\calFesbA)$, $\Zz^{\perp} \mVp{0}(\calFesbA)$, and $\Zz^\perp \mbV_{p,\perp}^1(\calFesbA)$ as in~\eqref{eq:def_ZZ}--\eqref{eq:def_ZZ_perp}. 
We will invoke the following discrete Poincar\'e inequalities on boundary extended stars. Their proofs are similar to the discrete Poincar\'e inequalities established on extended stars in the bulk in \cite{EGPV_HO:24,PGEV_Poinc:26} and are omitted for brevity. 

\begin{proposition}[Discrete Poincar\'e inequalities on boundary extended stars] \label{prop:discP_star}
For all $\sigma \in \Delta_h^{\partial}$, we have the discrete Poincar\'e inequalities:
\begin{subequations}
\begin{alignat}{2}
\|u\|_{L^2(\esb(\sigma))} &\lesssim h_{\sigma} \| \sgrad u \|_{\bL^2(\esb(\sigma))},  \quad && \forall u \in \Zz^{\perp} V_p^0(\calFesb),   \label{onto0}\\
\|\bu\|_{\bL^2(\esb(\sigma))} &\lesssim  h_{\sigma} \|\scurl \bu\|_{L^2(\esb(\sigma))}, \quad && \forall  \bu \in \Zz^{\perp} \bV_p^1(\calFesb),\label{onto1}
\end{alignat}
\end{subequations}
and
\begin{subequations}
\begin{alignat}{2}
 \|u\|_{L^2(\esb(\sigma))} &\lesssim h_{\sigma} \| \srot u \|_{\bL^2(\esb(\sigma))},  \quad && \forall u \in \Zz^{\perp} \mVp{0}(\calFesbA),   \label{onto0m}\\
 \|\bu\|_{\bL^2(\esb(\sigma))} &\lesssim h_{\sigma} \|\sdive \bu\|_{L^2(\esb(\sigma))}, \quad && \forall  \bu \in \Zz^{\perp} \mbV_{p,\perp}^1(\calFesbA). \label{onto1m}
\end{alignat}
\end{subequations}
\end{proposition}

\subsection{Lowest-order basis functions}

The lowest-order basis functions are chosen to be the Whitney forms attached to all the geometric entities lying on the boundary. Thus, we set
\begin{equation}
B_{0,r}^l := W_r^l, \qquad \forall r\in I_0^l:=\Delta_h^{l,\partial}, \qquad \forall l\in\{0{:}2\},
\end{equation}
so that the enumeration index $r$ corresponds to a boundary vertex for $l=0$, to a boundary edge for $l=1$ and to a boundary face for $l=2$. Notice that we indeed have $\tr^l(V_0^l) = \Span_{r\in I_0^l} (\tr^l(B_{0,r}^l))$ in agreement with~\eqref{eq:tr_Vl_span}. Moreover, recalling~\eqref{eq:W_iota} and~\eqref{CW}, the assumptions~\eqref{eq:ass_B0} hold true with
\begin{equation}
\kappa_0^{r,r'} := \iota_{r',r}, \qquad \sigma_{0,r}^l:=r, \qquad
\beta_{0,r}^l:=\frac32-l,
\end{equation}
where the quantities $\iota_{r',r}$ are the incidence numbers defined in 
Section~\ref{sec:Whitney}.

\subsection{Construction of the lowest-order boundary weights}

The construction of the lowest-order boundary weights essentially follows the ideas of the construction in the bulk from \cite{EGPV_HO:24}. For completeness, we provide some details. Let $\mu$ be the globally continuous function defined on $\Gamma$, that is piecewise affine on the boundary mesh $\FhBA$, vanishes on $\partial f$ for every $f \in \Fh^\partial$, and takes the value one at its barycenter. For all $\sigma \in \Delta_h^\partial$, we let 
\begin{equation} \label{eq:def_mu}
\mu_{\sigma} := \chi_{\esb(\sigma)} \mu,
\end{equation}
where $\chi_{\esb(\sigma)}$ is the characteristic function of $\esb(\sigma)$. 

The lowest-order boundary weights are constructed sequentially, starting from $l=0$, then $l=1$, and finally $l=2$. The construction needs that we prove along the way one of the required properties of these weights. We leave the verification of the remaining properties to the end of this section.

\subsubsection{Construction of $\zeta_{0,v}^0$ and $\bY_{0,v}^0$ for all $v \in \Vh^\partial$}

Let $\eta_{v}^0:=|\esb(v)|^{-1}\chi_{\esb(v)}$ and define $\psi_{v}^0 \in \Zz^{\perp} V_p^0(\calFevb)$ as the solution to
\begin{equation}\label{605}
( \mu_{v} \, \sgrad \psi_{v}^0, \sgrad u)_{\esb(v)} = \phi_v^\partial(u) - (\eta^0_v, u )_{\esb(v)}, \quad  \forall u \in   \Zz^{\perp} V_p^0(\calFevb).
\end{equation}
This problem is well-posed owing to the Poincar\'e inequality~\eqref{onto0}. 
We now define
\begin{equation}\label{defZ0}
\zeta_{0,v}^0:=\eta_v^0- \sdive \big(\mu_v \, \sgrad \, \psi_v^0 \big) \quad \text{in $\esb(v)$},
\end{equation}
and extend $\zeta_{0,v}^0$ by zero outside $\esb(v)$. 
Notice that $\zeta_{0,v}^0 \in V_p^2(\calFevbA)$.
Finally, we define $\bY_{0,v}^0$ as the element of $\bV_p^2(\ThWF)$ obtained by setting its boundary degrees of freedom so that $\tr^2(\bY_{0,v}^0)|_{\esb(v)} = \zeta_{0,v}^0$ and all its bulk (Worsey--Farin) degrees of freedom to zero. Hence, $\supp(\bY_{0,v}^0) \subset \clos(\es(v))$.

\begin{lemma}[Link to lowest-order boundary dof]
\textup{(i)} We have
\begin{equation}\label{805}
(\mu_{v} \, \sgrad \psi_v^0, \sgrad u)_{\esb(v)}=\phi_v^\partial(u) - (\eta_v^0, u )_{\esb(v)}, \quad\forall u \in V_p^0(\calFevb).
\end{equation}
\textup{(ii)}
The following holds:
\begin{equation} \label{eq:zeta_dof_0}
(\zeta_{0,v}^0 , \tr^0(u))_\Gamma = \phi_v(u), \quad  \forall u\in V^0_p.
\end{equation}
\end{lemma}

\begin{proof}
\textup{(i)} We need to show that~\eqref{605} also holds true for all $u\in \Zz V_p^0(\calFevb)$. This follows from the fact that both the left-hand side and the right-hand side vanish when $u\in \Zz V_p^0(\calFevb)$.

\textup{(ii)} Let $u\in V^0_p$. The identity follows from
\begin{align*}
(\zeta_{0,v}^0 , \tr^0(u))_\Gamma    
    &= (\eta_v^0 , \tr^0(u))_{\esb(v)}  - (\sdive(\mu_v \sgrad \psi_v^0) , \tr^0(u))_{\esb(v)} \\
    &= (\eta_v^0 , \tr^0(u))_{\esb(v)}  + (\mu_v \sgrad (\psi_v^0) , \sgrad(\tr^0(u)))_{\esb(v)}\\
    &= \phi_v^\partial(\tr^0(u)) = \phi_v(u),
\end{align*}
where we used~\eqref{805} since $\tr^0(u)|_{\esb(v)} \in V_p^0(\calFevb)$.
\end{proof}

\subsubsection{Construction of $\bzeta_{0,e}^1$ and $\bY_{0,e}^1$ for all $e\in\Eh^\partial$}

Since $\esb(v)\subset\esb(e)$ for all $v\in\Ve$ and $\zeta_{p,v}^0$ is supported on $\clos(\esb(v))$, we obtain
\begin{equation}
\int_{\esb(e)} \sum_{v\in\Ve} \iota_{ev} \zeta_{0,v}^0
= \sum_{v\in\Ve} \iota_{ev} \int_{\esb(v)} \zeta_{0,v}^0 = \sum_{v\in\Ve} \iota_{ev} = 0,
\end{equation}
where the second equality follows from~\eqref{eq:zeta_dof_0} (take $u$ constant equal to one on $\esb(v)$ to infer that $\int_{\esb(v)} \zeta_{0,v}^0=1$), and the last equality follows from~\eqref{eq:GRAD}.
Hence, by the exactness of the discrete sequence~\eqref{localcomplexm}, there exists $\beeta_e^1 \in \Zz^{\perp} \mbVpp{1}(\calFeebA)$ such that 
\begin{equation}\label{eta1def}
 -\sdive \, \beeta_e^1= \sum_{v\in\Ve} \iota_{ev} \zeta_{0,v}^0 \quad \text{ on $\esb(e)$}.
\end{equation} 
We extend $\beeta_e^1$ by zero outside $\esb(e)$.

Next, we define $\bpsi_e^1 \in \Zz^{\perp} \bV_p^1(\calFeeb)$ such that
\begin{equation}\label{605.1}
( \mu_{e} \, \scurl \bpsi_e^1, \scurl \bu)_{\esb(e)} =\phi_e^\partial(\bu) - ( \beeta_e^1, \bu )_{\esb(e)}, \quad  \forall \bu \in   \Zz^{\perp} \bV_p^1(\calFeeb).
\end{equation}
This problem is well-posed owing to the Poincar\'e inequality~\eqref{onto1}.
Next, we define 
\begin{equation}\label{defzeta1}
\bzeta_{0,e}^1:=\beeta_e^1 + \srot \big(\mu_e \, \scurl \bpsi_e^1 \big) \quad \text{in $\esb(e)$},
\end{equation}
and extend $\bzeta_{0,e}^1$ by zero outside $\esb(e)$. Notice that $\bzeta_{0,e}^1 \in \mbVpp{1}(\calFeebA)$.  
Finally, we define $\bY_{0,e}^1$ as the element of $\bV_p^1(\ThWF)$ obtained by setting its boundary degrees of freedom so that $\tr_\perp^1(\bY_{0,e}^1)|_{\esb(e)} = \bzeta_{0,e}^1$ and all its bulk (Worsey--Farin) degrees of freedom to zero. Hence, $\supp(\bY_{0,e}^1) \subset \clos(\es(e))$.

\begin{lemma}[Link to lowest-order boundary dof]
\textup{(i)} We have
\begin{equation}\label{805.1}
(\mu_{e} \, \scurl \bpsi_e^1, \scurl \bu)_{\esb(e)}=\phi_e^\partial(\bu) - (\beeta_e^1, \bu )_{\esb(e)}, \quad  \forall \bu \in \bV_p^1(\calFeeb).
\end{equation}
\textup{(ii)} The following holds:
\begin{equation} \label{eq:zeta_dof_1}
(\bzeta_{0,e}^1 , \tr^1(\bu))_\Gamma = \phi_e(\bu), \quad  \forall \bu\in \bV^1_p.
\end{equation}
\end{lemma}

\begin{proof}
\textup{(i)} We need to show that \eqref{805.1} also holds true for all 
$\bu \in \Zz \bV_p^1(\calFeeb)$. Since $\scurl \bu=0$, the exactness of the discrete sequence~\eqref{localcomplex} implies that $\bu=\sgrad m$ for some $m \in V_p^0(\calFeeb)$. We extend $m$ by zero outside $\esb(e)$. We observe that
\begin{alignat*}{2}
(\beeta^1(e), \sgrad m )_{\esb(e)}&= -( \sdive \beeta_e^1,  m )_{\esb(e)} \quad && \text{integration by parts, $\beeta^1(e)\in \mbVpp{1}(\calFeebA)$ } \\
&=  \sum_{v\in\Ve} \iota_{ev} (\zeta_{0,v}^0, m )_{\esb(e)} \quad && \text{by~\eqref{eta1def}} \\
&= \sum_{v\in\Ve} \iota_{ev} \phi^\partial_v(m)   \quad && \text{by~\eqref{eq:zeta_dof_0}, $m\in \tr^0(V_p^0)$} \\
&= \phi_e^\partial(\sgrad m), \quad && \text{by~\eqref{eq:iota_ev_bnd}}.
\end{alignat*}
This shows that \eqref{805.1} also holds true for all $\bu \in \Zz \bV_p^1(\calFeeb)$.

\textup{(ii)} Let $\bu\in\bV_p^1$. The identity follows from
\begin{align*}
(\bzeta_{0,e}^1 , \tr^1(\bu))_\Gamma
&= (\beeta_e^1 , \tr^1(\bu))_{\esb(e)}  + (\srot(\mu_e \scurl \bpsi_e^1) , \tr^1(\bu))_{\esb(e)} \\
&= (\beeta_e^1 , \tr^1(\bu))_{\esb(e)}  + (\mu_e \scurl \bpsi_e^1 , \scurl(\tr^1(\bu)))_{\esb(e)}\\
&= \phi_e^\partial(\tr^1(\bu)) = \phi_e(\bu),
\end{align*}
where we used \eqref{805.1} since $\tr^1(\bu)|_{\esb(e)} \in \bV_p^1(\calFeeb)$.
\end{proof}

\subsubsection{Construction of $\zeta_{0,f}^2$ and $Y_{0,f}^2$ for all $f\in\Fh^\partial$}

Combining~\eqref{eta1def} and~\eqref{defzeta1}, we infer that 
\begin{equation} \label{eq:sdiv_zeta1}
\sdive \bzeta_{0,e}^1 = 
- \sum_{v\in\Ve} \iota_{ev}\zeta_{0,v}^0. 
\end{equation}
This implies that
\begin{align}
\sdive \left( \sum_{e\in\Ef} \iota_{f e} \bzeta_{0,e}^1\right) = -\sum_{e\in\Ef}\sum_{v\in\Ve}\iota_{f e}\iota_{ev}\zeta_{0,v}^0 = -\sum_{v\in\Vf}\sum_{e\in\Ev\cap\Ef} \iota_{f e}\iota_{ev}\zeta_{0,v}^0 = 0,
\end{align}
where we used~\eqref{eq:CURL_GRAD} in the last equality. Hence, by the exactness of the discrete sequence~\eqref{localcomplexm} and the Poincar\'e inequality~\eqref{onto0m}, there exists $\eta_f^2 \in \Zz^{\perp} \mVp{0}(\calFefbA) = \mVp{0}(\calFefbA)$ such that 
\begin{equation}\label{eta2def}
\srot \eta_f^2= \sum_{e\in\Ef} \iota_{f e} \bzeta_{0,e}^1 \quad \text{ on $\esb(f)$}.
\end{equation}
We then set 
\begin{equation}\label{defZ3}
\zeta_{0,f}^2:=\eta_f^2 \quad \text{in $\esb(f)$},
\end{equation}
and extend $\zeta_{0,f}^2$ by zero outside $\esb(f)$.
Notice that $\zeta_{0,f}^2\in \mVp{0}(\calFefbA)$.
Finally, we define $Y_{0,f}^2$ as the element of $V_p^0(\ThWF)$ obtained by setting its boundary degrees of freedom so that $\tr^0(Y_{0,f}^2)|_{\esb(f)} = \zeta_{0,f}^2$ and all its bulk (Worsey--Farin) degrees of freedom to zero. Hence, $\supp(Y_{0,f}^2) \subset \clos(\es(f))$.

\begin{lemma}[Link to lowest-order boundary dof]
The following holds:
\begin{equation} \label{eq:zeta_dof_2}
(\zeta_{0,f}^2 , \tr^2(\bu))_\Gamma = \phi_f(\bu), \quad  \forall \bu\in \bV^2_p.
\end{equation}
\end{lemma}

\begin{proof}
Let $\bu\in \bV^2_p$. By the exactness of the discrete sequence~\eqref{localcomplex}, there exists $\bbm \in \bV_p^1(\calFefb)$ such that $\scurl \bbm=\tr^2(\bu)|_{\esb(f)}$. Notice that $\bbm = \tr^1(\tilde{\bbm})|_{\esb(f)}$ with $\tilde{\bbm}\in \bV_p^1$. We obtain
\begin{alignat*}{2}
( \zeta_{0,f}^2, \scurl \bbm )_{\esb(f)} &=
( \srot \zeta_{0,f}^2, \bbm )_{\esb(f)} \quad && \text{by \eqref{eq:scurl_adjoint}} \\
&= \sum_{e\in\Ef} \iota_{f e}( \bzeta_{0,e}^1, \tr^1(\tilde{\bbm}) )_{\esb(f)}   \quad && \text{by~\eqref{eta2def}, $\bbm = \tr^1(\tilde{\bbm})|_{\esb(f)}$} \\
&= \sum_{e\in\Ef} \iota_{f e} \phi_e^\partial(\bbm) \quad && \text{by~\eqref{eq:zeta_dof_1}, $\phi_e(\tilde{\bbm}) = \phi_e^\partial(\bbm)$} \\
&= \phi_f^\partial(\scurl \bbm) \quad && \text{by~\eqref{eq:iota_fe_bnd}}.
\end{alignat*}
This proves that $( \zeta_{0,f}^2, \scurl \bbm )_{\esb(f)} = 
\phi_f^\partial(\tr^2(\bu)) = \phi_f(\bu)$.
\end{proof}

\subsection{Properties of the lowest-order boundary weights}

\begin{lemma}[Link to lowest-order boundary dofs]
The following holds for all $r\in I_0^l$ and all $l\in\{0{:}2\}$:
\begin{equation} \label{eq:zeta_dof_all}
(\zeta_{0,r}^l , \tr^l(u))_\Gamma = \phi_r(u), \quad  \forall u\in V^l_p.
\end{equation}
\end{lemma}

\begin{proof}
The identity~\eqref{eq:zeta_dof_all} is established in~\eqref{eq:zeta_dof_0} for $l=0$, 
in~\eqref{eq:zeta_dof_1} for $l=1$, and in~\eqref{eq:zeta_dof_2} for $l=2$. 
\end{proof}

\begin{lemma}[Fulfillment of~\eqref{eq:ass_zeta0}]
The boundary weights $\{\zeta_{0,r}^l\}_{l\in\{0{:}2\}}$ satisfy~\eqref{eq:ass_zeta0}.
\end{lemma} 

\begin{proof}
Recall that $B_{0,r}^l := W_r^l$ for all $r\in I_0^l:=\Delta_h^{l,\partial}$ and all $l\in\{0{:}2\}$, and that $\kappa_0^{r,r'} := \iota_{r',r}$, $\sigma_{0,r}^l:=r$, $\beta_{0,r}^l:=\frac32-l$.

Proof of~\eqref{eq:Kron_zeta}. We need to prove that $(\zeta_{0,r}^l,\tr^l(B_{0,r'}^l))_\Gamma = \delta_{r,r'}$ for all $r,r'\in I_0^l$. This is an immediate consequence of~\eqref{eq:zeta_dof_all}.

Proof of~\eqref{deltazeta1}-\eqref{deltazeta2}. We need to prove that $-\sdive (\bzeta_{0,r'}^1) = \sum_{r \in I_0^0} \kappa_0^{r',r} \zeta_{0,r}^0$ for all $r'\in I_0^1$, and that
$\srot (\zeta_{0,r'}^2) = \sum_{r\in I_0^1} \kappa_0^{r',r} \bzeta_{0,r}^1$ for all $r'\in I_0^2$. 
Since the $\kappa_0^{r',r}$ are the incidence numbers, these identities are nothing but~\eqref{eq:sdiv_zeta1} and~\eqref{eta2def} (since $\eta^2_f=\zeta_{0,f}^2$).

Proof of~\eqref{eq:support_Y0}. We need to prove that, for all $l\in\{0{:}2\}$ and 
all $r\in I_0^l$, $\supp(Y_{0,r}^l) \subseteq \clos(\es(r))$ and
$h_{r} \|d^{2-l} Y_{0,r}^l\|_{L^2(\Omega)} \lesssim \|Y_{0,r}^l\|_{L^2(\Omega)}  \lesssim h_r^{l-\frac12}$. The support property is part of the construction, and the bound $h_{r} \|d^{2-l} Y_{0,r}^l\|_{L^2(\Omega)} \lesssim \|Y_{0,r}^l\|_{L^2(\Omega)}$ follows from an inverse inequality and the shape-regularity of the mesh. It remains to show that 
\begin{equation} \label{eq:bnd_zeta_l}
\|\zeta_{0,r}^l\|_{L^2(\esb(r))} \lesssim h_r^{l-1}, \quad \forall l\in\{0{:}2\},
\end{equation} 
since a scaling argument readily gives $\|Y_{0,r}^l\|_{L^2(\Omega)} \lesssim h_r^{\frac12} \|\zeta_{0,r}^l\|_{L^2(\esb(r))}$.

Proof of~\eqref{eq:bnd_zeta_l} for $l=0$. We first observe that 
$\| \eta_v^0\|_{L^2(\esb(v))}\lesssim h_v^{-1}$.
Moreover, using~\eqref{605}, the Cauchy--Schwarz inequality and the inverse inequality $\|\psi_v^0\|_{L^\infty(\esb(v))} \lesssim h_v^{-1} \|\psi_v^0\|_{L^2(\esb(v))}$, we obtain 
\begin{align*}
\| \sgrad \psi_v^0\|_{\bL^2_{\textsc{t}}(\esb(v))}^2 &\lesssim (\mu_{v} \, \sgrad \psi_v^0, \sgrad \psi_v^0)_{\esb(v)} \\
&=\phi_v^\partial(\psi_v^0) - ( \eta_v^0, \psi_v^0)_{\esb(v)} \\
&\le \|\psi_v^0\|_{L^\infty(\esb(v))} + \| \eta_v^0\|_{L^2(\esb(v))} \|\psi_v^0\|_{L^2(\esb(v))} \\
&\lesssim h_v^{-1}\|\psi_v^0\|_{L^2(\esb(v))}.
\end{align*}
Invoking the Poincar\'e inequality~\eqref{onto0}
to bound $\|\psi_v^0\|_{L^2(\esb(v))}$ gives
$\| \sgrad \psi_v^0\|_{\bL_{\textsc{t}}^2(\esb(v))} \lesssim 1$.
Invoking an inverse estimate, we obtain
\begin{equation*}
  \| \sdive \big(\mu_v \, \sgrad \psi_v^0 \big)\| _{L^2(\esb(v))}\lesssim h_v^{-1} \| \mu_v \, \sgrad  \psi_v^0\|_{\bL^2(\esb(v))} \lesssim h_v^{-1}.
\end{equation*}
Combining the above estimates proves that $\|\zeta_{0,v}^0\|_{L^2(\esb(v))} \lesssim h_v^{-1}$.

Proof of~\eqref{eq:bnd_zeta_l} for $l=1$. The Poincar\'e inequality~\eqref{onto1m}, the triangle inequality, the above bound on $\|\zeta_{0,v}^0\|_{L^2(\esb(v))}$, and the shape-regularity of the mesh give
\[
\| \beeta_e^1\|_{\bL^2(\esb(e))} \lesssim h_e  \sum_{v\in\Ve} \|\zeta_{p,v}^0\|_{L^2(\esb(v))} \lesssim 1.
\]
Using~\eqref{605.1} and proceeding as above, we obtain 
\begin{align*}
\| \scurl \bpsi_e^1\|_{L^2(\esb(e))}^2 &\lesssim (\mu_e\scurl  \bpsi_e^1,\scurl  \bpsi_e^1)_{\esb(e)} \\
&=\phi_e^\partial(\bpsi_e^1) -(\beeta_e^1, \bpsi_e^1)_{\esb(e)} \\
&\le h_e \|\bpsi_e^1\|_{\bL^\infty(\es(e))} + \| \beeta_e^1\|_{\bL_{\textsc{t}}^2(\esb(e))} \|\bpsi_e^1\|_{\bL^2_{\textsc{t}}(\esb(e))} \\
&\lesssim \|\bpsi_e^1\|_{\bL^2_{\textsc{t}}(\esb(e))}.
\end{align*}
Invoking the Poincar\'e inequality~\eqref{onto1} gives $\| \scurl \bpsi_e^1\|_{L^2(\esb(e))} \lesssim h_e$ and invoking an inverse estimate, we obtain 
\begin{equation*}
  \| \srot \big(\mu_e \, \scurl \bpsi_e^1 \big)\| _{\bL^2_{\textsc{t}}(\esb(e))}\lesssim  1.
\end{equation*}
Combining the above estimates proves that $\|\bzeta_{p,e}^1\|_{\bL^2_{\textsc{t}}(\esb(e))} \lesssim 1$.

Proof of~\eqref{eq:bnd_zeta_l} for $l=2$. Invoking the Poincar\'e inequality~\eqref{onto0m}, the triangle inequality, the above bound on $\|\bzeta_{p,e}^1\|_{\bL^2_{\textsc{t}}(\esb(e))}$, and the shape-regularity of the mesh, we infer that
\begin{equation}\label{eta2boundprel}
\| \zeta_{0,f}^2\|_{L^2(\esb(f))} \lesssim h_f \sum_{e\in\Ef}  \|\bzeta_{p,e}^1\|_{\bL^2_{\textsc{t}}(\esb(e))} \lesssim h_f.
\end{equation}
This concludes the proof of~\eqref{eq:bnd_zeta_l}.

Proof of~\eqref{eq:l=2pou}. We need to prove that $g:=\sum_{f \in \mathcal{F}_h^\partial} \tr^3(\dive(\bW_f)) \zeta_{0,f}^2 = 1$ on $\Gamma$. 
We first show that $\srot(g) = \bzero$. To this end, we first observe using \eqref{eta2def} that
\begin{align*}
    \srot(g) &= \sum_{f\in\Fh^{\partial}} \tr^3(\dive(\bW_f)) \srot(\zeta_{0,f}^2)\\
    &= \sum_{f\in\Fh^{\partial}} \tr^3(\dive(\bW_f)) \sum_{e\in\Ef}\iota_{fe}\bzeta_{0,e}^1 = \sum_{e\in\Eh^{\partial}} \left(\sum_{f\in\Fe\cap \Fh^{\partial}} \iota_{fe} \tr^3(\dive(\bW_f))\right)\bzeta_{0,e}^1.
\end{align*}
Moreover, using the divergence theorem, and that $\tr^2(\bW_f) = 0$ for all $f\in \mFh$, followed by \eqref{eq:W_iota} and \eqref{eq:DIV_CURL}, we have for all $e\in\Eh^\partial$,
\begin{align*}
    \sum_{f\in\Fe\cap \Fh^{\partial}} \iota_{fe} \tr^3(\dive(\bW_f)) &= \sum_{f\in\Fe\cap \Fh^{\partial}} \iota_{fe} (\tr^2(\bW_f), 1)_\Gamma = \sum_{f\in\Fe} \iota_{fe} (\tr^2(\bW_f), 1)_\Gamma\\
    &= \sum_{f\in\Fe} \iota_{fe} \tr^3(\dive(\bW_f)) = \sum_{f\in\Fe} \iota_{fe} \sum_{\tau\in\Tf} \iota_{\tau f} \tr^3(W_\tau)\\
    &= \sum_{\substack{\tau\in\Th \\ e\subset \tau}}\left( \sum_{f\in\Fe\cap \Ft} \iota_{\tau f}\iota_{fe}\right) \tr^3(W_\tau) = 0.
\end{align*}
Hence, $\srot(g) = \bzero$. Let $\{\Gamma_j\}_{j\in\{0{:}J\}}$ be the connected components of $\Gamma$. Since $\srot(g)=\bzero$, there exists a constant $c_j$ so that $g|_{\Gamma_j}=c_j$ for all $j\in\{0{:}J\}$. Let $j\in\{0{:}J\}$ and let us show that $c_j=1$. There exists $\bu_j\in \bV_0^2\subset \bV^2_p$ so that $\tr^2(\bu)|_{\Gamma_j}=1$ and $\tr^2(\bu)|_{\Gamma_{j'}}=0$ for all $j'\ne j$. Repeated applications of the divergence theorem give
\begin{align*}
    c_j\, \tr^3(\dive(\bu_j)) = (g, \tr^2(\bu_j))_\Gamma &= \sum_{f\in\Fh^\partial}\tr^3(\dive(\bW_f))\phi_f(\bu_j) \\
&= \sum_{f\in\Fh}\tr^3(\dive(\bW_f)) \phi_f(\bu_j) = \tr^3(\dive(\bu_j)),
\end{align*}
where the second equality follows from the definition of $g$ and~\eqref{eq:zeta_dof_2}. Since we chose $\bu_j$ so that $\tr^2(\bu)|_{\Gamma_j}=1$, we have $\tr^3(\dive(\bu_j)) = |\Gamma_j| \neq 0$, and so we conclude that $c_j=1$.
\end{proof}

\section{Construction of the higher-order basis functions and boundary weights} 
\label{sec:high_order}

In this section, we address the construction of the higher-order basis functions and boundary weights. We only outline the main ideas as they are adapted from \cite[Section 5]{Arn_Guz_loc_stab_L2_com_proj_21}, which dealt with the bulk setting, to the present boundary setting. We assume $p\ge1$ throughout this section (recall that $P_+^l=0$ if $p=0$, see Remark~\ref{rem:p=0}). 

\subsection{Higher-order basis functions}

For all $l\in\{0{:}2\}$, we set 
\begin{equation} \label{eq:def_V+_dofs}
    V_{+,p}^l := \{v\in V_p^l: \phi_\sigma(v) = 0, \forall \sigma\in\Delta_h^l\}.
\end{equation}
It is clear that $V_p^l=V_0^l\oplus V_{+,p}^l$, so that the decomposition~\eqref{eq:linear_dec} holds true.

To define the higher-order basis functions, we consider boundary geometric objects $\sigma \in \Delta_h^\partial$ with $\dim(\sigma)\in \mathfrak{d}^l$, with $\mathfrak{d}^0:=\{1{:}\min(2,p)\}$, $\mathfrak{d}^1:=\{1{:}2\}$, and $\mathfrak{d}^2:=\{2\}$, and we set $\mathfrak{D}^{l,\partial}:=\bigcup_{l'\in\mathfrak{d}^l} \Delta_h^{l',\partial}$. It is useful to define traces on boundary geometric objects of functions belonging to the trace spaces. Specifically, we set
\begin{subequations}
\begin{alignat}{5}
&\tr^0_e(\mu):=\mu|_e,\;&&\forall e\in \Eh^\partial,\quad&&
\tr^0_f(\mu):=\mu|_f,\;&&\forall f\in\Fh^\partial,\qquad&&
\forall \mu\in \tr^0(V_p^0), \\
&\tr^1_e(\bmu):=\bmu|_e{\cdot}\bt_e,\;&&\forall e\in \Eh^\partial,\quad&&
\tr^1_f(\bmu):=\bmu|_f,\;&&\forall f\in\Fh^\partial,\qquad&&
\forall \bmu\in \tr^1(\bV_p^1), \\
&&&&&\tr^2_f(\mu):=u|_f,\;&&\forall f\in\Fh^\partial,\qquad&&
\forall \mu\in \tr^2(\bV_p^2).
\end{alignat} 
\end{subequations}
Notice that we only define $\tr^l_\sigma$ with $\sigma\in \Delta_h^{l',\partial}$ when $l'\ge \max(l,1)$.
For all $l\in\{0{:}2\}$ and all $\sigma\in \Delta_h^{l',\partial}$ with $l'\ge \max(l,1)$, we set
\begin{equation}
V_p^l(\sigma) := \tr^l_\sigma ( \tr^l (V_p^l)),
\end{equation}
as well as
\begin{equation}
\mV_p^l(\sigma) := \begin{cases}
\{\varphi \in V_p^l(\sigma)\::\: \varphi|_{\partial\sigma}=0\}&\text{if $l<\dim(\sigma)$},\\
\{\varphi \in V_p^l(\sigma)\::\: (\varphi,1)_{\sigma}=0\}&\text{if $l=\dim(\sigma)$}.
\end{cases}
\end{equation}
We equip the space $\mV_p^l(\sigma)$ with the following inner product:
\begin{equation}
\bll \varphi,\psi \brr_\sigma^l := (\varphi,\psi)_{\sigma} + (d^l_\sigma\varphi,d^l_\sigma\psi)_{\sigma},
\end{equation}
where the right-hand side consists of $L^2(\sigma)$- or $\bL^2(\sigma)$-inner products, and the exterior derivatives on $\sigma$ are defined as
\begin{subequations} \begin{alignat}{4}
&d^0_e(\varphi):=(\grad\,u)|_e{\cdot}\bt_e,\quad&&\forall e\in\Eh^\partial,\quad
&&\forall \varphi:=\tr^0_e(\tr^0(u)),\quad &&u\in V_p^0, \\
&d^0_f(\varphi):=(\sgrad u)|_f,\quad&&\forall f\in \Fh^\partial,\quad
&&\forall \varphi:=\tr^0_f(\tr^0(u)),\quad &&u\in V_p^0, \\
&d^1_f(\bvarphi):=(\scurl\bu)|_f,\quad&&\forall f\in \Fh^\partial,\quad
&&\forall \bvarphi:=\tr^1_f(\tr^1(\bu)),\quad &&\bu\in \bV_p^1,
\end{alignat} \end{subequations}
together with $d^1_e\equiv0$ and $d^2_f\equiv0$. We notice that $d^l_\sigma(\mV^l_p(\sigma)) \subset \mV^{l+1}_p(\sigma)$, with the convention $\mV^3_p(\sigma):=\{0\}$. We also have
\begin{equation}\label{eq:simplex_commute}
\tr_\sigma^{l+1}(\tr^{l+1}(d^{l} u)) = d_\sigma^{l}(\tr_\sigma^{l}(\tr^l(u))), \quad \forall u\in V^l_p.
\end{equation}

For all $l\in\{0{:}2\}$ and all $\sigma \in \mathfrak{D}^{l,\partial}$, we consider a specific basis of $\mV_p^l(\sigma)$, $\Bh_p^l(\sigma)$, which we partition as 
$\Bh_p^l(\sigma) = \Zh\Bh_p^l(\sigma) \cup \Zhp\Bh_p^l(\sigma)$, with the following properties:
\begin{subequations} \begin{alignat}{2}
&\bll g,g' \brr^l_\sigma = \delta_{g,g'},\quad&&\forall g,g' \in \Bh_p^l(\sigma), \label{eq:delta_gg'}\\
&d^l_\sigma g = 0,\quad&&\forall g\in \Zh\Bh_p^l(\sigma),\\
&d^l_\sigma g \in \Zh\Bh_p^{l+1}(\sigma),\quad&&\forall g\in \Zhp\Bh_p^l(\sigma), \; \forall l\in\{0{:}1\}. 
\end{alignat} \end{subequations}
For $l=2$, $\Zhp\Bh_p^2(\sigma)=\emptyset$ for all $\sigma\in\mathfrak{D}^{2,\partial}=\Fh^\partial$, so that $\Bh_p^2(\sigma)=\Zh\Bh_p^2(\sigma)$.

We are now ready to define the higher-order basis functions and to review their properties. For all $l\in\{0{:}2\}$, the indexing set is
\begin{equation}
I_{+}^l:= \{ r:= (\sigma,g) \::\: \sigma \in \mathfrak{D}^{l,\partial}, \; g\in \Bh_p^l(\sigma)\}.
\end{equation}
The higher-order basis functions $\{B_{+,r}^l\}_{r\in I_{+}^l}\subset V_{+,p}^l$ are defined in such a way that, for all $r:= (\sigma,g)\in I_{+}^l$,
\begin{subequations} \label{eq:E_props}\begin{align}
&\tr^l_\sigma(\tr^l(B_{+,r}^l)) = g, \label{eq:trtr_Esig} \\
&d^lB_{+,r}^l = \begin{cases}
0,&\text{if $g\in \Zh\Bh_p^l(\sigma)$},\\
B_{+,(\sigma,d^l_\sigma g)}^{l+1},&\text{if $g\in \Zhp\Bh_p^l(\sigma)$}, \end{cases}
\label{eq:dE_sigg}\\
&\supp(B_{+,r}^l) \subseteq \clos(\st(\sigma)),\label{eq:E_supp}\\
&\|B_{+,r}^l\|_{L^2(\st(\sigma))} \lesssim h_\sigma^{\beta_{+,r}^l},\label{eq:E_bound}
\end{align} \end{subequations}
where $\beta_{+,r}^l := \frac12(3-\dim(\sigma))$ if $g\in \Zh\Bh_p^l(\sigma)$
and $\beta_{+,r}^l := \frac12(1-\dim(\sigma))$ if $g\in \Zhp\Bh_p^l(\sigma)$. 
We refer the reader to \cite[Section 5]{Arn_Guz_loc_stab_L2_com_proj_21} for explicit definitions of the higher-order basis functions satisfying the properties \eqref{eq:E_props}.
We note the following consequence of~\eqref{eq:delta_gg'} and~\eqref{eq:trtr_Esig}:
\begin{equation} \label{eq:tr_Esig}
\bll \tr^l_{\sigma'}(\tr^l(B_{+,r}^l)),g' \brr^l_{\sigma'} = \delta_{r,r'} = \delta_{\sigma,\sigma'}
\delta_{g,g'},\quad\forall r:=(\sigma,g)\in I_+^l,\; \forall r':=(\sigma',g') \in I_{+}^l.
\end{equation}
The family $\{B_{+,r}^l\}_{r\in I_+^l}$ is linearly independent and we have $\tr^l(V_{+,p}^l)=
\Span_{r\in I_+^l} (\tr^l(B_{+,r}^l))$, as requested in~\eqref{eq:tr_Vl_span}. Moreover,
Assumption~\ref{ass:B+} holds true. Indeed, the relation to differential 
operators~\eqref{eq:d_varphi_relation+} is satisfied with $\kappa_+^{r',r}:=\delta_{\sigma,\sigma'}
\delta_{d_\sigma^lg,g'}$, for all $r:=(\sigma,g)\in I_+^l$ and all $r':=(\sigma',g')
\in I_+^{l+1}$, and the properties~\eqref{eq:support_B+} related to the support and norm
are satisfied with $\sigma_{+,r}^l:=\sigma\in \Delta_h^\partial$ and $\beta_{+,r}^l$ defined below~\eqref{eq:E_bound}, for all $r:=(\sigma,g)\in I_+^l$. 
Finally, the mean-zero divergence property~\eqref{eq:mean_zero_div_l=2} when $l=2$ follows from \eqref{eq:dE_sigg} since $\Zhp\Bh_p^2(\sigma) = \emptyset$.

\subsection{Higher-order boundary weights}

We now turn our attention to the higher-order boundary weights $\{\zeta_{+,r}^l\}_{r\in I_+^l}$ for all $l\in\{0{:}2\}$. We split $I_+^l$ into two disjoint subsets $I_+^l = I_{\Zh}^l \cup I_{\Zhp}^l$ where $I_{\Zh}^l := \{(\sigma,g) \in I_+^l \::\: g \in \Zh\Bh_p^l(\sigma)\}$ and $I_{\Zhp}^l := \{(\sigma,g) \in I_+^l \::\: g \in \Zhp\Bh_p^l(\sigma)\}$ (notice that $I_{\Zhp}^2=\emptyset$). The reason is that the construction of the boundary weights is different according to whether $r\in I_{\Zh}^l$ or $r\in I_{\Zhp}^l$.

Let us start by defining the boundary weights $\{\zeta_{+,r}^l\}_{r\in I_{\Zh}^l}$ for all $l\in\{0{:}2\}$. For all $\sigma\in\mathfrak{D}^{l,\partial}$, we define the space
\begin{equation}\label{eq:local_spaces_M}
V_{+,p}^l(\calFstb) := \{v|_{\stb(\sig)} \::\: v=\tr^l(u),\; u\in V_{+,p}^l\}.
\end{equation}
We also let $\mu_+(\sigma) := \chi_{\stb(\sigma)} \mu$, where  $\mu$ is as in \eqref{eq:def_mu}.
Thus, the function $\mu_+(\sigma)$ is supported in $\stb(\sigma)$ and vanishes on all edges and vertices of $\calFstb$. For all $r:=(\sigma,g)\in I_{\Zh}^l$, let $\beta_{r}^l  \in V_{+,p}^l(\calFstb)$ solve the following problem:
\begin{equation} \label{eq:beta_prop}
(\mu_+(\sigma) \beta_{r}^l, v )_{\stb(\sigma)} = \bll g,\tr_\sigma^l(v)\brr_{\sigma}^l, \quad \forall v \in V_{+,p}^l(\calFstb),
\end{equation}
and set
\begin{equation} \label{eq:zeta_z}
    \zeta_{+,r}^l := \mu_+(\sigma) \beta_{r}^l.
\end{equation}
Furthermore, we define the boundary weights $\{\zeta_{+,r}^l\}_{r\in I_{\Zhp}^l}$ for all $l\in\{0{:}1\}$ by setting
\begin{subequations}\label{zeta_zperp}
\begin{alignat}{2}
\zeta_{+,r}^0 &:= -\sdive (\zeta_{+, (\sigma, d_{\sigma}^0g)}^1), \quad && \forall r:=(\sigma, g) \in I_{\Zhp}^0,\label{zeta_zperp1}\\
\bzeta_{+,r}^1 & := \srot (\zeta_{+, (\sigma, d_{\sigma}^1 g)}^2), \quad && \forall r:=(\sigma, g) \in I_{\Zhp}^1.\label{zeta_zperp2}
    \end{alignat}
\end{subequations}
Altogether, we have $\supp(\zeta_{+,r}^l) \subseteq \clos(\stb(\sigma))$ for all $r:=(\sigma, g)\in I_{+}^l$, $\zeta_{+,r}^l \in V_{p+1}^l(\calFstbA)$ for all $r\in I_{\Zh}^l$ and $\zeta_{+,r}^l \in V_{p}^l(\calFstbA)$ for all $r\in I_{\Zhp}^l$. We then extend the boundary weights by zero outside of their supports to all of $\Gamma$. Finally, for all $l\in\{0{:}2\}$ and all $r \in I_+^l$, we define the extension $Y_{+,r}^l \in V_{p+1}^l(\ThWF)$ by setting all its boundary degrees of freedom so that $\tr^l(Y_{+,r}^l)=\zeta_{+,r}^l$ and all its bulk (Worsey--Farin) degrees of freedom to zero. Hence, for all $r:=(\sigma, g)\in I_{+}^l$, $\supp(Y_{+,r}^l)\subset \clos(\st(\sigma))$.

\begin{lemma}[Fulfillment of~\eqref{eq:ass_zeta+}]
The boundary weights $\{\zeta_{+,r}^l\}_{r\in I_+^l}$ satisfy~\eqref{eq:ass_zeta+}.
\end{lemma} 

\begin{proof}
Proof of~\eqref{eq:Kron_zeta+}. We only detail the case $l=1$, as the cases $l=0$ and $l=2$ follow from analogous arguments. We need to show that $(\bzeta_{+,r}^1,\tr^1(\bB_{+,r'}^1))_{\Gamma} = \delta_{r,r'}$ for all $r,r'\in I_+^1$. We will do this by showing that 
\begin{equation} \label{eq:zeta_tr1}
(\bzeta_{+,r}^1,\tr^1(\bu))_\Gamma = \bll g,\tr_\sigma^1(\tr^1(\bu))\brr_{\sigma}^1, \quad \forall \bu \in \bV_{+,p}^1, \quad r:=(\sigma,g)\in I_{+}^1,
\end{equation} 
so that invoking \eqref{eq:tr_Esig} gives the result. Assume first that $r\in I_{\Zh}^{1}$. Since $\tr^1(\bu)|_{\stb(\sigma)} \in \bV_{+,p}^1(\calFstb)$, we infer from~\eqref{eq:beta_prop} that
\begin{align*}
(\bzeta_{+,r}^1,\tr^1(\bu))_\Gamma = (\bzeta_{+,r}^1,\tr^1(\bu))_{\stb(\sigma)} = \bll g,\tr_\sigma^1(\tr^{1}(\bu))\brr_{\sigma}^1.
\end{align*}
This proves~\eqref{eq:zeta_tr1} for all $r\in I_{\Zh}^l$. Assume now that $r\in I_{\Zhp}^l$.
We have 
\begin{alignat*}{2}
(\bzeta_{+,r}^1,\tr^1(\bu))_{\Gamma}
&= (\srot(\zeta_{+,(\sigma,d_\sigma^1 g)}^2),\tr^1(\bu))_{\stb(\sigma)}, \quad &&\text{by \eqref{zeta_zperp2}}\\
&= (\zeta_{+,(\sigma,d_\sigma^1 g)}^2,\scurl(\tr^1(\bu)))_{\stb(\sigma)}, \quad &&\text{by \eqref{eq:scurl_adjoint}}\\
&= (\zeta_{+,(\sigma,d_\sigma^1 g)}^2,\tr^2(\curl(\bu)))_{\stb(\sigma)}, \quad &&\text{by \eqref{eq:trace_d_commuting}}\\
&= \bll d_{\sigma}^1 g, \tr_{\sigma}^2(\tr^2(\curl(\bu)))\brr_{\sigma}^2 \quad &&\text{by \eqref{eq:beta_prop}}\\
&= \bll d_{\sigma}^1 g,d_\sigma^1(\tr_\sigma^1(\tr^1(\bu)))\brr_{\sigma}^2 \quad &&\text{by \eqref{eq:simplex_commute}}\\
&= \bll g,\tr_\sigma^1(\tr^1(\bu))\brr_{\sigma}^1 \quad &&\text{since $g \in \Zhp\Bh_p^1(\sigma)$}.
\end{alignat*}

Next, we show \eqref{eq:Kron_zeta0+}, i.e., $(\zeta_{0,r}^l,\tr^l(B_{+,r'}^l))_\Gamma = 0$ for all $(r,r')\in I_0^l\times I_+^l$. This simply follows from \eqref{eq:zeta_dof_all} and the fact that $\phi_\sigma^l(B_{+,r'}^l)=0$ since $B_{+,r'}^l\in V_{+,p}^l$ satisfies~\eqref{eq:def_V+_dofs}.

We now show \eqref{deltazeta1+}-\eqref{deltazeta2+}, and we only detail the proof of~\eqref{deltazeta1+} as the proof of~\eqref{deltazeta2+} follows from similar arguments. We need to prove that 
\[
-\sdive (\bzeta_{+,r'}^1) = \sum_{r \in I_+^0} \kappa_+^{r',r} \zeta_{+,r}^0,\quad \forall r'\in I_+^1,
\]
where $\kappa_+^{r',r}=\delta_{\sigma,\sigma'} \delta_{d^0_\sigma g,g'}$ for all $r:=(\sigma,g)\in I_+^0$ and $r':=(\sigma',g')\in I_+^{1}$. If $r'\in I_{\Zh}^1$, then the above identity directly follows from~\eqref{zeta_zperp1}. Instead, if $r':=(\sigma',g')\in I_{\Zhp}^1$, then using \eqref{zeta_zperp2}, we have $\bzeta_{+,r'}^1 = \srot \bzeta_{+,(\sigma', d_\sigma^1 g')}^2$, and thus $-\sdive (\bzeta_{+,r'}^1) = -\sdive (\srot (\zeta_{+,(\sigma', d_\sigma^1 g')}^2)) = 0$, which matches the right-hand side of~\eqref{deltazeta1+} since in this case all the coefficients $\kappa_+^{r',r}$ vanish.

Finally, we turn to \eqref{eq:Ybounds+}. We have already seen that $Y_{+,r}^l$ is supported in $\clos(\st(\sigma))$ for all $r:=(\sigma,g)\in I_+^l$. Moreover, invoking standard scaling arguments shows that $\|\zeta_{+,r}^l\|_{L^2(\stb(\sigma))} \lesssim h_{\sigma}^{\frac{\dim(\sigma)-2}{2}}$ when $r\in I_{\Zh}^l$ and $\|\zeta_{+,r}^l\|_{L^2(\stb(\sigma))} \lesssim h_{\sigma}^{\frac{\dim(\sigma)}{2}}$ when $r\in I_{\Zhp}^l$. The rest of the proof proceeds as the proof of \eqref{eq:support_Y0}.
\end{proof}


\bibliographystyle{acm_mod}
\bibliography{references}

\appendix

\end{document}